\documentclass[10pt,letterpaper]{article}
\usepackage[utf8]{inputenc}
\usepackage{amsmath}
\usepackage{amsfonts}
\usepackage{amssymb}
\usepackage{graphicx}
\usepackage{mathrsfs}
\usepackage{upref,amsthm,amsxtra,exscale}
\usepackage{cite}
\usepackage[colorlinks=true,urlcolor=blue,
citecolor=red,linkcolor=blue,linktocpage,pdfpagelabels,
bookmarksnumbered,bookmarksopen]{hyperref}

\newtheorem{theorem}{Theorem}[section]
\newtheorem{corollary}[theorem]{Corollary}

\newtheorem{lemma}[theorem]{Lemma}
\newtheorem{proposition}[theorem]{Proposition}

\numberwithin{equation}{section}

\author{Mónica Clapp\footnote{M. Clapp was partially supported by CONACYT grant 237661 (Mexico) and UNAM-DGAPA-PAPIIT grant IN100718 (Mexico).} \quad and \quad Jorge Faya\footnote{J. Faya was supported by a postdoctoral fellowship under CONACYT grant 237661 (Mexico).}}
\title{Multiple solutions to a weakly coupled purely critical elliptic system in bounded domains}
\date{\today}

\begin{document}
\maketitle

\begin{abstract}
We study the weakly coupled critical elliptic system
\begin{equation*}
\begin{cases}
-\Delta u=\mu_{1}|u|^{2^{*}-2}u+\lambda\alpha |u|^{\alpha-2}|v|^{\beta}u & \text{in }\Omega,\\
-\Delta v=\mu_{2}|v|^{2^{*}-2}v+\lambda\beta |u|^{\alpha}|v|^{\beta-2}v & \text{in }\Omega,\\
u=v=0 & \text{on }\partial\Omega,
\end{cases}
\end{equation*}
where $\Omega$ is a bounded smooth domain in $\mathbb{R}^{N}$, $N\geq 3$, $2^{*}:=\frac{2N}{N-2}$ is the critical Sobolev exponent, $\mu_{1},\mu_{2}>0$, $\alpha, \beta>1$, $\alpha+\beta =2^{*}$ and $\lambda\in\mathbb{R}$.

We establish the existence of a prescribed number of fully nontrivial solutions to this system under suitable symmetry assumptions on $\Omega$, which allow domains with finite symmetries, and we show that the positive least energy symmetric solution exhibits phase separation as $\lambda\to -\infty$.

We also obtain existence of infinitely many solutions to this system in $\Omega=\mathbb{R}^N$.    \vspace{4pt}


\end{abstract}

\section{Introduction}

We consider the weakly coupled critical elliptic system
\begin{equation} \label{eq:system}
\begin{cases}
-\Delta u=\mu_{1}|u|^{2^{*}-2}u+\lambda\alpha |u|^{\alpha-2}|v|^{\beta}u,\\
-\Delta v=\mu_{2}|v|^{2^{*}-2}v+\lambda\beta |u|^{\alpha}|v|^{\beta-2}v,\\
u,v\in D^{1,2}_0(\Omega),
\end{cases}
\end{equation}
where $\Omega$ is either a bounded smooth domain in $\mathbb{R}^{N}$ or $\Omega=\mathbb{R}^{N}$, $N\geq 3$, $2^{*}:=\frac{2N}{N-2}$ is the critical Sobolev exponent, $\mu_{1},\mu_{2}>0$, $\lambda\in\mathbb{R}$, $\alpha, \beta>1$ and $\alpha+\beta =2^{*}$. 

This type of systems arises, e.g., in the Hartree-Fock theory for double condensates, that is, Bose-Einstein condensates of two different hyperfine states which overlap in space; see \cite{egbb}. The sign of $\mu_{i}$ reflects the interaction of the particles within each single state. This interaction is attractive if $\mu_{i}>0$. The sign of $\lambda$ reflects the interaction of particles in different states. It is attractive if $\lambda>0$ and it is repulsive if $\lambda<0$. If the condensates repel, they separate spatially. This phenomenon is called phase separation and has been described in \cite{t}. The system \eqref{eq:system} is called \emph{cooperative} if $\lambda>0$ and it is called \emph{competitive} if $\lambda<0$.

Weakly coupled elliptic systems have attracted considerable attention in recent years, and there is an extensive literature on subcritical systems, specially on the cubic system (where $\alpha=\beta=2$ and $2^*$ is replaced by $4$) in dimensions $N\leq 3$; we refer to \cite{so} for a detailed account of the achievements in the subcritical case. In contrast, there are still few results for critical systems.

When $\lambda=0$ the system \eqref{eq:system} reduces to the problem
\begin{equation} \label{eq:bc}
-\Delta w=|w|^{{2}^{*}-2}w,\text{\qquad}w\in D^{1,2}_0(\Omega),
\end{equation}
which has been extensively studied in the last decades. W.Y. Ding showed in \cite{d} that it has infinitely many solutions if $\Omega=\mathbb{R}^{N}$. It is also well known that it does not have a nontrivial solution if $\Omega$ is strictly starshaped and $\Omega\neq\mathbb{R}^{N}$. A remarkable result by Bahri and Coron \cite{bc} establishes the existence of a positive solution to \eqref{eq:bc} in every bounded smooth domain with nontrivial $\mathbb{Z}_2$-homology.

The existence of multiple solutions to \eqref{eq:bc} in a bounded domain is, to a great extent, an open problem. It is shown in \cite{gmp} that, in a bounded punctured domain, the number of sign-changing solutions to the problem \eqref{eq:bc} becomes arbitrarily large, as the hole schrinks. Multiplicity in bounded symmetric domains was studied in \cite{cf0}, where it is shown that the number of sign-changing solutions to \eqref{eq:bc} increases, as the cardinality of the orbits increases.

Note that, if $w$ solves \eqref{eq:bc}, then $(\mu_{1}^{(2-N)/4}w,0)$ and $(0,\mu_{2}^{(2-N)/4}w)$ solve the system \eqref{eq:system} for every $\lambda$. Solutions of this type are called \emph{semitrivial}. We are interested in \emph{fully nontrivial} solutions to \eqref{eq:system}, i.e., solutions where both components, $u$ and $v$ are nontrivial. A solution is said to be \emph{synchronized} if it is of the form $(sw,tw)$ with $s,t\in\mathbb{R}$, and it is called \emph{positive} if $u\geq0$ and $v\geq0$. 

Our aim is to present some results regarding the existence and multiplicity of fully nontrivial solutions to the critical system \eqref{eq:system}.

In the cooperative case, we make the following additional assumption in dimensions $\leq5$:

\begin{itemize}
\item[$(A)$] If $\lambda>0$ and, either $N=3$ or $4$, or $N=5$ and $\alpha\not\in (\frac{4}{3},2)$, then there exists $r\in(0,\infty)$ such that
$$\mu_1 r^{2^*-2} + \lambda\alpha r^{\alpha -2}-\lambda\beta r^\alpha -\mu_2=0.$$
\end{itemize}
With this assumption, every nontrivial solution to the problem \eqref{eq:bc} gives rise to a fully nontrivial synchronized solution of the system \eqref{eq:system} with $\lambda>0$; see Section \ref{sec:cooperative}. In particular, we obtain the following version of the Bahri-Coron theorem.

\begin{theorem} \label{thm:bc}
Let $\Omega$ be a bounded smooth domain in $\mathbb{R}^N$, $\lambda>0$, and assume $(A)$. If $\widetilde{H}_*(\Omega;\mathbb{Z}_2)\neq 0$, then the system \eqref{eq:system} has a positive fully nontrivial synchronized solution.
\end{theorem}

Here $\widetilde{H}_*(\,\cdot\,;\mathbb{Z}_2)$ stands for reduced singular homology with $\mathbb{Z}_2$-coefficients. The special case of punctured domains was treated in \cite{ppw}.

In the competitive case the situation is quite different. In fact, there exists $\lambda_*<0$ such that the system \eqref{eq:system} does not have a synchronized solution if $\lambda<\lambda_*$; see \cite[Proposition 2.3]{cp}. 

It is an open question whether Theorem \ref{thm:bc} is true or not for $\lambda<0$. Answering this question is not easy. The only results that we know of for the system \eqref{eq:system} with $\lambda<0$ in a bounded domain are those recently obtained by Pistoia and Soave in \cite{ps}, where they established existence of multiple fully nontrivial solutions in bounded punctured domains of dimension $3$ or $4$ using the Lyapunov-Schmidt reduction method. This method cannot be applied in higher dimensions, due to the low regularity of the interaction term. 

In this paper we shall consider symmetric domains, and we will obtain results in every dimension.

Let $O(N)$ denote, as usual, the group of linear isometries of $\mathbb{R}^{N}.$ If $G$ is a closed subgroup of $O(N)$, we write $Gx:=\{gx:g\in G\}$ for the $G$-orbit of $x\in\mathbb{R}^{N}$ and $\#Gx$ for its cardinality. A subset $X$ of $\mathbb{R}^{N}$ is said to be $G$-invariant if $Gx\subset X$ for every $x\in X$, and a function $u:X\to\mathbb{R}$ is called $G$-invariant if it is constant on every $G$-orbit of $X$.

Fix a closed subgroup $\Gamma$ of $O(N)$ and a nonempty $\Gamma$-invariant bounded smooth domain $\Theta$ in $\mathbb{R}^N$ such that the $\Gamma$-orbit of every point $x\in\Theta$ has positive dimension. Then, we prove the following result.

\begin{theorem} \label{thm:bounded}
Fix $\Gamma$ and $\Theta$ as above, and assume $(A)$. Then, for any given $n\in\mathbb{N}$, there exists $\ell_n >0$, depending on $\Gamma$ and $\Theta$, such that, if $\Omega$ is a bounded smooth domain in $\mathbb{R}^N$ which contains $\Theta$, $\Omega$ is $G$-invariant for some closed subgroup $G$ of $\Gamma$, and
$$\min_{x\in\bar{\Omega}} \#Gx > \ell_n,$$
then, for each $\lambda \neq0$, the system \eqref{eq:system} has at least $n$ nonequivalent fully nontrivial $G$-invariant solutions $(u_1,v_1),\ldots,(u_n,v_n)$ which satisfy
$$\int_{\Omega}(|\nabla u_j|^2+|\nabla v_j|^2) \leq \left(1+\min_{x\in\bar{\Omega}} \#Gx\right)\mu_0^{\frac{2-N}{2}}S^{\frac{N}{2}}\quad\text{for each }\;j=1,\ldots,n,$$
where $\mu_0:=\max\{\mu_1,\mu_2\}$ and $S$ is the best Sobolev constant for the embedding $D^{1,2}(\mathbb{R}^N)\hookrightarrow L^{2^*}(\mathbb{R}^N)$. Moreover, $(u_1,v_1)$ is positive and has least energy among all $G$-invariant solutions.
\end{theorem}

A similar statement for the single equation \eqref{eq:bc} was proved in \cite{cf0}. Nonequivalent means that, if $i\neq j$, then $(u_i,v_i)$ is different from $\pm(u_j,v_j)$, $\pm(u_j,-v_j)$, $\pm(v_j,u_j)$ and $\pm(v_j,-u_j)$. 

The solutions given by Theorem \ref{thm:bounded} are synchronized if $\lambda>0$. In contrast, as shown in \cite[Proposition 2.3]{cp}, there exists $\lambda_* <0$, depending on $\mu_{1},\mu_{2},\alpha, \beta$, such that no solution is synchronized if $\lambda<\lambda_*$. 

If every $G$-orbit in $\Omega$ is infinite, Theorem \ref{thm:bounded} yields the following result.

\begin{corollary} \label{cor:bounded}
Assume $(A)$. If $\Omega$ is a $G$-invariant bounded smooth domain in $\mathbb{R}^N$ and the $G$-orbit of every point $x\in\Omega$ has positive dimension, then, for each $\lambda \neq 0$, the system \eqref{eq:system} has infinitely many fully nontrivial $G$-invariant solutions.
\end{corollary}

Furthermore, Theorem \ref{thm:bounded} yields multiple solutions even in domains which have finite $G$-orbits. Let us give an example. If $2k\leq N$ we write $\mathbb{R}^N\equiv\mathbb{C}^k\times\mathbb{R}^{N-2k}$ and the points in $\mathbb{R}^N$ as $(z,x)$ with $z\in\mathbb{C}^k$, $x\in\mathbb{R}^{N-2k}$. The group $\mathbb{S}^1:=\{\mathrm{e}^{\mathrm{i}\vartheta}:\vartheta\in[0,2\pi)\}$ of unit complex numbers acts on $\mathbb{R}^N$ by $\mathrm{e}^{\mathrm{i}\vartheta}(z,x):=(\mathrm{e}^{\mathrm{i}\vartheta}z,x)$. The $\mathbb{S}^1$-orbit of $(z,x)$ is homeomorphic to $\mathbb{S}^1$ iff $z\neq 0$. For each $m\geq 2$, let $\mathbb{Z}_m$ be the cyclic subgroup of $\mathbb{S}^1$ generated by $\mathrm{e}^{2\pi\mathrm{i}/m}$. For these group actions we obtain the following result.

\begin{corollary} \label{cor:example}
Let $\Theta$ be an $\mathbb{S}^1$-invariant bounded smooth domain contained in $(\mathbb{C}^k\smallsetminus\{0\})\times\mathbb{R}^{N-2k}$, and assume $(A)$. Then, for each $\lambda \neq0$, the system \eqref{eq:system} has infinitely many $\mathbb{S}^1$-invariant fully nontrivial solutions in $\Theta$.

Moreover, for any given $n$ there exists $\ell_n>0$ such that, if $m>\ell_n$, then the system \eqref{eq:system} has $n$ nonequivalent $\mathbb{Z}_m$-invariant fully nontrivial solutions in every $\mathbb{Z}_m$-invariant bounded smooth domain $\Omega$ which contains $\Theta$ and does not intersect $\{0\}\times\mathbb{R}^{N-2k}$.
\end{corollary}

In particular, if $N=2k$, we may take $\Theta$ to be an annulus. Then, Corollary \ref{cor:example} yields examples of annular domains, with a hole of arbitrary size, in which the system \eqref{eq:system} has a prescribed number of solutions for any $\lambda\neq 0$. A similar statement for the single equation \eqref{eq:bc} was proved in \cite{cpa}.

Our next result says that the $G$-invariant least energy solutions given by Theorem \ref{thm:bounded} exhibit phase separation as $\lambda\to -\infty$.

\begin{theorem} \label{thm:separation}
Let $\Omega$ be a $G$-invariant bounded smooth domain. Assume that, for some sequence $(\lambda_{k})$ with $\lambda_{k}\to -\infty$, there exists a positive fully nontrivial least energy $G$-invariant solution $(u_{k},v_{k})$ to the system \eqref{eq:system} with $\lambda=\lambda_{k}$, such that
$$\int_{\Omega}(|\nabla u_k|^2+|\nabla v_k|^2) \leq \left(1+\min_{x\in\bar{\Omega}} \#Gx\right)\mu_0^{\frac{2-N}{2}}S^{\frac{N}{2}}\quad\text{for every }\;k\in \mathbb{N},$$
where $\mu_0:=\max\{\mu_1,\mu_2\}$ and $S$ is the best Sobolev constant for the embedding $D^{1,2}(\mathbb{R}^N)\hookrightarrow L^{2^*}(\mathbb{R}^N)$. 

Then, after passing to a subsequence, we have that $u_{k}\to u_{\infty}$ and $v_{k}\to v_{\infty}$ strongly in $D_0^{1,2}(\Omega)$, the functions $u_{\infty}$ and $v_{\infty}$ are continuous and $G$-invariant, $u_{\infty}\geq 0$, $v_{\infty}\geq 0$, $u_{\infty}v_{\infty}\equiv 0$, $u_{\infty}$ solves the problem
$$-\Delta u=\mu_{1}|u|^{{2}^*-2}u,\qquad u\in D_{0}^{1,2}(\Omega_{1}),$$
and $v_{\infty}$ solves the problem
$$-\Delta v=\mu_{2}|v|^{{2}^*-2}v,\qquad v\in D_{0}^{1,2}(\Omega_{2}),$$
where $\Omega_{1}:=\{x\in\Omega:u_{\infty}(x)>0\}$ and $\Omega_{2}:=\{x\in\Omega:v_{\infty}(x)>0\}$. Moreover, $\Omega_{1}$ and $\Omega_{2}$ are $G$-invariant, $\Omega_{1}\cap\Omega_{2}=\emptyset$ and $\overline{\Omega_{1}\cup\Omega_{2}}=\overline{\Omega}.$
\end{theorem}

Note that, if $\Omega$ is an annulus  and the solutions $(u_k,v_k)$ given by Theorem \ref{thm:separation} are radial, then the limiting domains $\Omega_1$ and $\Omega_2$ must be annuli.

Phase separation for weakly coupled subcritical systems in a bounded domain was established in \cite{ctv1,ctv2} via minimization on a suitable constraint. Critical Brezis-Nirenberg type systems, obtained by adding a linear term to both equations in \eqref{eq:system}, have been recently treated in \cite{cz1,cz2,llw,pt}. For these systems, phase separation occurs in dimensions $N\geq 6$; see \cite{cz2}. 

Our final result concerns the case when the domain is the whole space $\mathbb{R}^N$. For the system \eqref{eq:system} in $\mathbb{R}^N$, with $\alpha=\beta$, it is shown in \cite{cz1,cz2} that there exists a positive fully nontrivial solution for all $\lambda>0$ if $N\geq 5$ and for a wide range of $\lambda>0$ if $N=4$. For $\lambda<0$ a positive fully nontrivial solution was exhibited in \cite{cp}, and infinitely many fully nontrivial solutions were obtained in \cite{cp} when $\mu_1=\mu_2$, $\alpha=\beta$ and $\lambda\leq\frac{\mu_1}{\alpha}$. These results are contained the following one.

\begin{theorem} \label{thm:entire}
Assume $(A)$. If $\Omega=\mathbb{R}^N$, then, for each $\lambda \neq 0$, the system  \eqref{eq:system} has infinitely many fully nontrivial solutions, which are not conformally equivalent, one of which is positive.
\end{theorem}

As in \cite{cp}, we use the conformal invariance of the system \eqref{eq:system} in $\mathbb{R}^N$ to prove Theorem \ref{thm:entire}. A different approach was used in \cite{glw} to establish the existence of positive multipeak solutions for $\lambda<0$ in dimension $N=3$. 

The positive entire solutions given by Theorem \ref{thm:entire} also exhibit phase separation as $\lambda\to -\infty$. This was shown in \cite[Theorem 1.2]{cp}. Moreover, a precise description of the limit domains $\Omega_1$ and $\Omega_2$ is given in \cite{cp}.

To obtain our results, we use variational methods. As in the case of the single equation \eqref{eq:bc}, the main difficulty is the lack of compactness of the functional associated to the system \eqref{eq:system}. We prove a representation theorem for $G$-invariant Palais-Smale sequences for this functional, which provides a full description of the loss of compactness in the presence of symmetries for every $\lambda\in\mathbb{R}$; see Theorem \ref{thm:struwe} below.

This paper is organized as follows. In Section \ref{sec:variational setting} we introduce the variational setting. Section \ref{sec:compactness} is devoted to the study of the loss of compactness in the presence of symmetries. Our main results are proved in Section \ref{sec:cooperative} in the cooperative case, and in Section \ref{sec:competitive} in the competitive case. Finally, in the Appendix we derive some Brezis-Lieb type results for the interaction term.

\section{The variational setting} \label{sec:variational setting}

Throughout this paper we assume that $\mu_{1},\mu_{2}>0$, $\lambda\in\mathbb{R}$, $\alpha, \beta>1$ and $\alpha+\beta =2^{*}$.

Let $\Omega$ be a smooth domain in $\mathbb{R}^N$, $N\geq 3$. As usual, $D^{1,2}_0(\Omega)$ denotes the closure of $\mathcal{C}_c^\infty(\Omega)$ in the space $D^{1,2}(\mathbb{R}^N):=\{u\in L^{2^*}(\mathbb{R}^N):\nabla u\in L^2(\mathbb{R}^N,\mathbb{R}^N)\}$ equipped with the norm
$$\|u\|:=\left(\int_{\mathbb{R}^N} |\nabla u|^{2}\right)^{1/2}.$$
Let $\mathscr{D}(\Omega):=D_{0}^{1,2}(\Omega)\times D_{0}^{1,2}(\Omega)$ with the product norm
$$\|(u,v)\|:=\left(\int_{\Omega} |\nabla u|^{2} +|\nabla v|^{2}\right)^{1/2}.$$
The solutions to the system \eqref{eq:system} are the critical points of the functional $E:\mathscr{D}(\Omega)\to \mathbb{R}$ defined by 
\begin{equation*}
E(u,v):=\frac{1}{2}\int_{\Omega}(|\nabla u|^{2}+|\nabla v|^{2})-\frac{1}{2^{*}}\int_{\Omega}(\mu_{1}|u|^{2^{*}}+\mu_{2}|v|^{2^{*}})-\lambda\int_{\Omega}|u|^{\alpha}|v|^{\beta}.
\end{equation*}
Since $\alpha, \beta >1$, this functional is of class $\mathcal{C}^{1}$ and its derivative is given by
\begin{equation*}
E'(u,v)(\varphi,\psi)=\partial_{u}E(u,v)\varphi+\partial_{v}E(u,v)\psi,
\end{equation*}
where
\begin{align*}
\partial_{u}E(u,v)\varphi &:=\int_{\Omega}\nabla u\cdot\nabla \varphi-\int_{\Omega}\mu_{1} |u|^{2^{*}-2}u\varphi-\lambda\alpha\int_{\Omega}|u|^{\alpha-2}|v|^{\beta}u\varphi, \\
\partial_{v}E(u,v)\psi &:=\int_{\Omega}\nabla v\cdot\nabla \psi-\int_{\Omega}\mu_{2} |v|^{2^{*}-2}v\psi-\lambda\beta\int_{\Omega}|u|^{\alpha}|v|^{\beta-2}v\psi.
\end{align*}
The fully nontrivial solutions to the system \eqref{eq:system} belong to the set
$$\mathcal{N}(\Omega):=\{(u,v)\in \mathscr{D}(\Omega): u\neq0, v\neq0, \, \partial_{u}E(u,v)u=0, \, \partial_{v}E(u,v)v=0 \}.$$

\begin{proposition} \label{prop:nehari}
If $\lambda <0$, then the following statements hold true:
\begin{enumerate}
\item[(a)] For every $(u,v)\in\mathcal{N}(\Omega)$, one has that
$$\mu_{1}^{-(N-2)/2}S^{N/2}\leq \|u\|^{2}, \qquad \mu_{2}^{-(N-2)/2}S^{N/2}\leq\|v\|^{2},$$
where $S$ is the best constant for the embedding $D^{1,2}(\mathbb{R}^{N})\hookrightarrow L^{2^{*}}(\mathbb{R}^{N}).$
\item[(b)] $\mathcal{N}(\Omega)$ is a closed $\mathcal{C}^{1}$-submanifold of $\mathscr{D}(\Omega)$ of codimension $2$. The tangent space to $\mathcal{N}(\Omega)$\ at the point $(u,v)$ is the orthogonal complement in $\mathscr{D}(\Omega)$ of the linear subspace generated by $\nabla F_1(u,v)$ and $\nabla F_2(u,v)$, where $F_1(u,v):=\partial_{u}E(u,v)u$ and $F_2(u,v):=\partial_{v}E(u,v)v$.
\item[(c)] $\mathcal{N}(\Omega)$ is a natural constraint for the functional $E$, i.e., a critical point of the restriction of $E$ to $\mathcal{N}(\Omega)$ is a critical point of $E$.
\end{enumerate}
\end{proposition}

\begin{proof}
The proof follows as in \cite[Proposition 2.1]{cp}, with minor modifications.
\end{proof}

\begin{proposition} \label{prop:nonmin} 
If $\lambda <0$, then
$$\inf_{(u,v)\in\mathcal{N}(\Omega)}E(u,v)=\frac{1}{N}(\mu_{1}^{-(N-2)/2}+\mu_{2}^{-(N-2)/2})S^{N/2}$$
and this value is never attained by $E$ on $\mathcal{N}(\Omega)$.

\end{proposition}

\begin{proof}
The argument used in \cite[Proposition 2.2]{cp} to prove this statement for $\mathbb{R}^N$ can be easily adapted to a general domain $\Omega$.
\end{proof}

As for the single equation \eqref{eq:bc}, one has the following nonexistence result. It is stated in \cite{ppw} for $\mu_1=1=\mu_2$ and $\lambda=\frac{1}{2^*}$, but the proof carries over to the general case. We give it here for the sake of completeness.

\begin{proposition} \label{prop:pohozhaev}
Let $\lambda\in\mathbb{R}$. If $\Omega \neq \mathbb{R}^N$ and $\Omega$ is strictly starshaped, then the system \eqref{eq:system} does not have a nontrivial solution.
\end{proposition}

\begin{proof}
Without loss of generality, we assume that $\Omega$ is strictly starshaped with respect to the origin. Let $(u,v)$ be a solution to the system \eqref{eq:system}. Multiplying the first equation in \eqref{eq:system} by $\nabla u \cdot x$ and the second one by $\nabla v \cdot x$ we get
\begin{align}
-\lambda\alpha |u|^{\alpha-2}|v|^{\beta}u(\nabla u \cdot x) &= (\Delta u + \mu_{1}|u|^{2^{*}-2}u)(\nabla u \cdot x), \label{eq:poho1} \\
-\lambda\beta |u|^{\alpha}|v|^{\beta-2}v(\nabla v \cdot x) &= (\Delta v + \mu_{2}|v|^{2^{*}-2}v)(\nabla v \cdot x). \label{eq:poho2}
\end{align}
Note that
$$\alpha |u|^{\alpha-2}|v|^{\beta}u(\nabla u \cdot x) + \beta |u|^{\alpha}|v|^{\beta-2}v(\nabla v \cdot x) = \mathrm{div}(|u|^\alpha |v|^\beta x) - N|u|^\alpha |v|^\beta .$$
So multiplying equations \eqref{eq:poho1} and \eqref{eq:poho2} by $\frac{1}{N}$, adding them up and integrating, we obtain the identity
\begin{align*}
\lambda\int_\Omega |u|^\alpha |v|^\beta =& \,\frac{1}{2^*}\int_\Omega (|\nabla u|^2 + |\nabla v|^2) - \frac{1}{2^*}\int_\Omega (\mu_1|u|^{2^*} + \mu_2|v|^{2^*})\\
& + \frac{1}{2N}\int_{\partial\Omega} \left(\left|\frac{\partial u}{\partial \nu}\right|^2 + \left|\frac{\partial v}{\partial \nu}\right|^2\right)(s \cdot \nu) \mathrm{d}s,
\end{align*}
where $\nu=\nu(s)$ is the outer unit normal to $\partial\Omega$ at $s$. As $(u,v)$ solves the system \eqref{eq:system}, this identity reduces to
$$\int_{\partial\Omega} \left(\left|\frac{\partial u}{\partial \nu}\right|^2 + \left|\frac{\partial v}{\partial \nu}\right|^2\right)(s \cdot \nu) \mathrm{d}s = 0.$$
Since $\Omega$ is strictly starshaped, this implies that $\frac{\partial u}{\partial \nu} = 0 = \frac{\partial v}{\partial \nu}$ on $\partial\Omega$ so, extending $u$ and $v$ by zero outside of $\Omega$, we obtain a solution to the system \eqref{eq:system} in the whole of $\mathbb{R}^N$ which vanishes in an open subset of $\mathbb{R}^N$. By the unique continuation principle, $u=0$ and $v=0$ in $\Omega$, as claimed.
\end{proof}

Propositions \ref{prop:nonmin} and \ref{prop:pohozhaev} showcase the lack of compactness of the functional $E$. Symmetries help restore compactness.

\section{Symmetries and compactness} \label{sec:compactness}

Let $G$ be a closed subgroup of $O(N)$. We will assume, from now on, that $\Omega$ is a $G$-invariant bounded smooth domain and we will look for $G$-invariant solutions to the system \eqref{eq:system}, i.e., solutions $(u,v)$ such that both components, $u$ and $v$, are $G$-invariant. Set 
$$\mathscr{D}(\Omega)^{G}:=\{(u,v)\in \mathscr{D}(\Omega):u\text{ and }v \text{ are }G\text{-invariant}\}.$$

The $G$-orbit $Gx$ of a point $x\in\mathbb{R}^{N}$ is $G$-homeomorphic to the homogeneous space $G/G_{x},$ where 
$$G_{x}:=\{g\in G:gx=x\}$$
is the isotropy group of $x.$ In particular, $\#Gx=|G/G_{x}|$, where $|G/K|$ denotes, as usual, the index of the subgroup $K$ in $G$. 

We will prove the following result.

\begin{theorem} \label{thm:struwe}
Let $\lambda\in\mathbb{R}$ and $((u_k,v_k))$ be a sequence such that
$$(u_k,v_k)\in \mathscr{D}(\Omega)^{G},\qquad E(u_k,v_k)\to c,\qquad \nabla E(u_k,v_k)\to 0.$$
Then, after passing to a subsequence, there exist a solution $(u,v)$ to the system \eqref{eq:system}, an integer $m\geq 0$ and, for each $j=1,...,m$, a closed subgroup $K_j$ of finite index in $G$, a sequence $(\xi_{j,k})$ in $\Omega$, a sequence $(\varepsilon_{j,k})$ in $(0,\infty)$, and a nontrivial (but possibly semitrivial) $K_j$-invariant solution $(\widetilde{u}_j, \widetilde{v}_j)$ to the limit system
\begin{equation} \label{eq:limitsystem}
\begin{cases}
-\Delta u=\mu_{1}|u|^{2^{*}-2}u+\lambda\alpha |u|^{\alpha-2}|v|^{\beta}u,\\
-\Delta v=\mu_{2}|v|^{2^{*}-2}v+\lambda\beta |u|^{\alpha}|v|^{\beta-2}v,\\
u,v\in D^{1,2}(\mathbb{R}^{N}),
\end{cases}
\end{equation}
with the following properties:
\begin{enumerate}
\item[(i)] $G_{\xi_{j,k}}=K_j$ for all $k\in\mathbb{N}$ and $j=1,...,m$.
\item[(ii)] $\varepsilon_{j,k}^{-1}\mathrm{dist}(\xi_{j,k},\partial\Omega)\to\infty$ and $\varepsilon_{j,k}^{-1}|g\xi_{j,k}-\tilde{g}\xi_{j,k}|\to\infty$ as $k\to\infty$, for any $g,\tilde{g}\in G$ with $g^{-1}\tilde{g}\notin K_j$ and each $j=1,...,m$.
\item[(iii)] $\lim\limits_{k\to\infty}\left\|(u_k,v_k) - (u,v) - \sum\limits_{j=1}^m \; \sum\limits_{[g]\in G/K_j} \, (\widetilde{u}_{j,k,g},\widetilde{v}_{j,k,g})\right\| = 0$, where
$$\widetilde{u}_{j,k,g}(x):=\varepsilon_{j,k}^{\frac{2-N}{2}}\widetilde{u}_j\left(\frac{g^{-1}x-\xi_{j,k}}{\varepsilon_{j,k}}\right),$$
$$\widetilde{v}_{j,k,g}(x):=\varepsilon_{j,k}^{\frac{2-N}{2}}\widetilde{v}_j\left(\frac{g^{-1}x-\xi_{j,k}}{\varepsilon_{j,k}}\right).$$
\item[(iv)] $c= E(u,v) + \sum\limits_{j=1}^m |G/K_j|E(\widetilde{u}_j,\widetilde{v}_j).$
\end{enumerate}
\end{theorem}

The rest of the section is devoted to the proof of this theorem. The following lemma will allow us to choose a $G$-orbit of concentration in a convenient way.

\begin{lemma} \label{lem:cf}
Given sequences $(\varepsilon_{k})$ in $(0,\infty)$ and $(\zeta_{k})$ in $\mathbb{R}^{N},$ there exist a sequence $(\xi_{k})$ in $\mathbb{R}^{N}$ and a closed subgroup $K$ of $G$ such that, after
passing to a subsequence, the following statements hold true:
\begin{itemize}
\item[(a)] The sequence $(\varepsilon_{k}^{-1}$\emph{dist}$(\xi_{k},G\zeta_{k}))$ is bounded.
\item[(b)] $G_{\xi_{k}}=K$ for all $k\in\mathbb{N}$.
\item[(c)] If $|G/K| <\infty$ then $\varepsilon_{k}^{-1}|g\xi_{k}-\tilde{g}\xi_{k}|\to\infty$ for any $g,\tilde{g}\in G$ with $g^{-1}\tilde{g}\notin K$.
\item[(d)] If $|G/K|=\infty$ then there is a closed subgroup $K'$ of $G$ such that $K\subset K'$,\, $|G/K'| =\infty$ and $\varepsilon_{k}^{-1}|g\xi_{k}-\tilde{g}\xi_{k}|\to\infty$ for any $g,\tilde{g}\in G$ with
$g^{-1}\tilde{g}\notin K'$.
\end{itemize}
\end{lemma}

\begin{proof}
See Lemma 3.3 in \cite{cf}.
\end{proof}

We will also need the following lemma.

\begin{lemma} \label{lem:hatS}
Set $\lambda^+:=\max\{\lambda,0\}$. Then,
$$\widehat{S}:=\inf_{\substack{(u,v)\in\mathscr{D}(\mathbb{R}^N)\\ (u,v)\neq (0,0)}}\frac{\|u\|^2 + \|v\|^2}{\left(\int_{\mathbb{R}^N}\left(\mu_1|u|^{2^*}+\mu_2|v|^{2^*} + 2^*\lambda^+ |u|^{\alpha}|v|^{\beta}\right)\right)^{2/2^*}}>0.$$
\end{lemma}

\begin{proof}
Let $((u_k,v_k))$ be a minimizing sequence for $\widehat{S}$ in $\mathscr{D}(\mathbb{R}^N)\smallsetminus \{(0,0)\}$. Without loss of generality, we may assume that $u_k\neq 0$ for all $k$. Let $t_k\in[0,\infty)$ be such that $|v_k|_{2^*} = t_k|u_k|_{2^*}$, where $|\cdot|_{2^*}$ is the norm in $L^{2^*}(\mathbb{R}^{N})$. 

Set $w_k:=u_k$ if $t_k=0$ and $w_k:=t_k^{-1}v_k$ if $t_k\neq 0$. Then, $w_k\in D^{1,2}(\mathbb{R}^{N})\smallsetminus\{0\}$ and $|w_k|_{2^*} = |u_k|_{2^*}$. Set $\bar{\mu}:=(\max\{\mu_1,\mu_2,2^*\lambda\})^{2/2^*}$
\begin{align*}
\widehat{S} + o(1) &\geq \frac{\|u_k\|^2 + \|v_k\|^2}{\bar{\mu}\left(\int_{\mathbb{R}^N}\left(|u_k|^{2^*}+|v_k|^{2^*} + |u_k|^{\alpha}|v_k|^{\beta}\right)\right)^{2/2^*}} \\
&= \frac{\|u_k\|^2}{\bar{\mu}\left(\int_{\mathbb{R}^N}\left(|u_k|^{2^*} +t_k^{2^*}|u_k|^{2^*} + t_k^\beta |u_k|^{\alpha}|w_k|^{\beta}\right)\right)^{2/2^*}} \\
&\qquad +\frac{t_k^2\|w_k\|^2}{\bar{\mu}\left(\int_{\mathbb{R}^N}\left(|w_k|^{2^*}+ t_k^{2^*}|w_k|^{2^*} + t_k^\beta |u_k|^{\alpha}|w_k|^{\beta}\right)\right)^{2/2^*}} \\
&\geq \frac{\|u_k\|^2}{\bar{\mu}(1+t_k^{2^*}+t_k^\beta)^{2/2^*}|u_k|_{2^*}^2} + \frac{t_k^2\|w_k\|^2}{\bar{\mu}(1+t_k^{2^*}+t_k^\beta)^{2/2^*}|w_k|_{2^*}^2} \\
&\geq \frac{1}{\bar{\mu}}\min_{t\geq 0}\frac{1+t^2}{(1+t^{2^*}+t^\beta)^{2/2^*}}S,
\end{align*}
where $S$ is the best constant for the embedding $D^{1,2}(\mathbb{R}^{N})\hookrightarrow L^{2^{*}}(\mathbb{R}^{N})$. Note that $t^\beta \leq 1+t^{2^*}$. Therefore,
$$\frac{1+t^2}{(1+t^{2^*}+t^\beta)^{2/2^*}} \geq \frac{1+t^2}{2^{2/2^*}(1+t^{2^*})^{2/2^*}} \geq \frac{1}{2^{2/2^*}},$$
for every $t\geq 0$. This completes the proof.
\end{proof}

The main step in the proof of Theorem \ref{thm:struwe} is given by the following result.

\begin{lemma} \label{lem:struwe} 
Let $\lambda\in\mathbb{R}$ and $((u_k,v_k))$ be sequence such that
$$(u_k,v_k)\in \mathscr{D}(\Omega)^{G},\qquad E(u_k,v_k)\to c,\qquad \nabla E(u_k,v_k)\to 0.$$
Assume further that $(u_k,v_k)\rightharpoonup (0,0)$ weakly in $\mathscr{D}(\Omega)$ but $(u_k,v_k)\not\to (0,0)$ strongly in $\mathscr{D}(\Omega)$. Then, after passing to a subsequence, there exist a closed subgroup $K$ of finite index in $G$, a sequence $(\xi_k)$ in $\Omega$, a sequence $(\varepsilon_k)$ in $(0,\infty)$, a nontrivial (but possibly semitrivial) $K$-invariant solution $(\widetilde{u}, \widetilde{v})$ to the limit system \eqref{eq:limitsystem}, and a sequence $((w_k,z_k))$ with the following properties:
\begin{enumerate}
\item[(i)] $G_{\xi_k}=K$ for all $k\in\mathbb{N}$.
\item[(ii)] $\varepsilon_k^{-1}\emph{dist}(\xi_k,\partial\Omega)\to\infty$ and $\varepsilon_{k}^{-1}|g\xi_{k}-\tilde{g}\xi_{k}|\to\infty$ for any $g,\tilde{g}\in G$ with $g^{-1}\tilde{g}\notin K$.
\item[(iii)] $\lim\limits_{k\to\infty}\left\|(u_k,v_k) - (w_k,z_k) - \sum\limits_{[g]\in G/K} \, (\widetilde{u}_{k,g},\widetilde{v}_{k,g})\right\| = 0$, where
$$\widetilde{u}_{k,g}(x):=\varepsilon_{k}^{\frac{2-N}{2}}\widetilde{u}\left(\frac{g^{-1}x-\xi_{k}}{\varepsilon_{k}}\right),\qquad\widetilde{v}_{k,g}(x):=\varepsilon_{k}^{\frac{2-N}{2}}\widetilde{v}\left(\frac{g^{-1}x-\xi_{k}}{\varepsilon_{k}}\right).$$
\item[(iv)] $c=\lim\limits_{k\to\infty}E(w_k,z_k) + |G/K|\,E(\widetilde{u},\widetilde{v})$.
\item[(v)] $(w_k,z_k)\in \mathscr{D}(\Omega)^{G}$,\; $(w_k,z_k)\rightharpoonup (0,0)$ weakly in $\mathscr{D}(\Omega)$,\; $E(w_k,z_k)\to c$ and $\nabla E(w_k,z_k)\to 0$.
\end{enumerate}
\end{lemma}

\begin{proof}
Passing to a subsequence, we may assume that $(u_k,v_k)\neq(0,0)$ for all $k$. Since
\begin{equation*}
\|(u_k,v_k)\|^{2}= N\left(E(u_k,v_k)-\frac{1}{2^{*}}E'(u_k,v_k)(u_k,v_k)\right) \leq Nc + 1 + \|(u_k,v_k)\|
\end{equation*}
for  $k$ large enough, the sequence $((u_k,v_k))$ is bounded in $\mathscr{D}(\Omega)$ and, thus, 
$\|(u_k,v_k)\|^{2}\to Nc$. As $(u_k,v_k)\not\to (0,0)$ strongly in $\mathscr{D}(\Omega)$, this implies that $c>0$. So we may fix $\delta$ with 
\begin{equation} \label{eq:delta}
0<2\delta<\min\{Nc,\, \widehat{S}^{\frac{N}{2}}\},
\end{equation}
where $\widehat{S}$ is as in Lemma \ref{lem:hatS}. Set 
$$f(s,t):=\mu_1|s|^{2^*}+\mu_2|t|^{2^*} + 2^*\max\{\lambda,0\}\, |s|^{\alpha}|t|^{\beta}.$$
Then, we have that
\begin{align*}
Nc+o(1)&=N\left(E(u_k,v_k)-\frac{1}{2}E'(u_k,v_k)(u_k,v_k)\right)\\
&=\int_{\Omega}\left(\mu_1|u_k|^{2^*}+\mu_2|v_k|^{2^*}\right) + 2^*\lambda \int_{\Omega} |u_k|^{\alpha}|v_k|^{\beta} \\
&\leq \int_{\Omega}f(u_k,v_k).
\end{align*}
As $\delta\in(0,\frac{Nc}{2})$, there exist bounded sequences $(\varepsilon_k)$ in $(0,\infty)$ and $(\zeta_k)$ in $\mathbb{R}^{N}$ such that, after
passing to a subsequence,
\begin{equation*}
\sup_{x\in\mathbb{R}^{N}}\int_{B_{\varepsilon_k}(x)}f(u_k,v_k) = \int_{B_{\varepsilon_k}(\zeta_k)}f(u_k,v_k) = \delta,
\end{equation*}
where $B_r(y):=\{x\in\mathbb{R}^N:|x-y|<r\}$. For these sequences we choose $K$ and $(\xi_k)$ as in Lemma \ref{lem:cf}. Then, $G_{\xi_k}=K$ and there exist $g_k\in G$ and a positive constant $C_{1}$ such that
$$\varepsilon_{k}^{-1}|\xi_{k}-g_k\zeta_{k}|=\varepsilon_{k}^{-1}\mathrm{dist}(\xi_{k},G\zeta_{k})<C_1$$
for all $k\in\mathbb{N}$. Therefore, $(\xi_k)$ is bounded and, since $u_k$ and $v_k$ are $G$-invariant, we get that
\begin{equation*}
\delta = \int_{B_{\varepsilon_k}(g\zeta_k)}f(u_k,v_k) \leq \int_{B_{C_0\varepsilon_k}(\xi_k)}f(u_k,v_k),
\end{equation*}
where $C_0:=C_1 +1$. This implies, in particular, that
\begin{equation} \label{eq:xi}
\mathrm{dist}(\xi_k,\bar{\Omega})<C_0\varepsilon_k.
\end{equation}

We claim that $|G/K|<\infty$. Otherwise, according to Lemma \ref{lem:cf}\emph{(d)}, for any $n\in\mathbb{N}$ there exist $g_1,...,g_n\in G$ such that $\varepsilon_k^{-1}|g_i\xi_k - g_j\xi_k|\to \infty$ for every $i\neq j$ and, hence, for $k$ large enough,
\begin{align*}
n\delta \leq \sum_{i=1}^n \int_{B_{C_0\varepsilon_k}(g_i\xi_k)}f(u_k,v_k)\leq \int_{\Omega}f(u_k,v_k)\leq \widehat{S}^{-1}(Nc+1),
\end{align*}
This is a contradiction.

Set $\Omega_k:=\{y\in\mathbb{R}^{N}:\varepsilon_k y+\xi_k\in\Omega\}$ and, for $y\in\Omega_k$, set
$$\widetilde{u}_k(y):=\varepsilon_k^{\frac{N-2}{2}}u_k(\varepsilon_k y+\xi_k), \qquad \widetilde{v}_k(y):=\varepsilon_k^{\frac{N-2}{2}}v_k(\varepsilon_k y+\xi_k).$$
Then, $(\widetilde{u}_k,\widetilde{v}_k)$ is bounded in $\mathscr{D}(\mathbb{R}^{N})$. So, passing to a subsequence, we get that $\widetilde{u}_k\rightharpoonup \widetilde{u}$ and $\widetilde{v}_k\rightharpoonup\widetilde{v}$ weakly in $D^{1,2}(\mathbb{R}^{N})$, $\widetilde{u}_k\rightarrow\widetilde{u}$ and $\widetilde{v}_k\rightarrow\widetilde{v}$ in $L_{\mathrm{loc}}^{2}(\mathbb{R}^{N})$, and $\widetilde{u}_k\rightarrow\widetilde{u}$ and $\widetilde{v}_k\rightarrow\widetilde{v}$ a.e. in $\mathbb{R}^{N}$. Moreover, for every $y\in\mathbb{R}^N$,
\begin{equation} \label{eq:sup}
\int_{B_{1}(y)}f(\widetilde{u}_k,\widetilde{v}_k) \leq \delta \leq \int_{B_{C_0}(0)}f(\widetilde{u}_k,\widetilde{v}_k).
\end{equation} 
Note also that, as $u_k$ and $v_k$ are $G$-invariant and $G_{\xi_k}=K$, we have that $\widetilde{u}_k$ and $\widetilde{u}_k$ are $K$-invariant. Hence, $\widetilde{u}$ and $\widetilde{v}$ are $K$-invariant.

For $\varphi\in \mathcal{C}_c^{\infty}(\mathbb{R}^N)$ we set $\varphi_k(x):=\varphi(\frac{x-\xi_k}{\varepsilon_k})$. Then, $(\varphi_k^2u_k)$ is bounded in $D^{1,2}_0(\Omega)$ and, hence,
\begin{align}
&\int_{\Omega_k}\nabla \widetilde{u}_k\cdot\nabla (\varphi^2 \widetilde{u}_k) - \mu_{1}\int_{\Omega_k}\varphi^2|\widetilde{u}_k|^{2^{*}} - \lambda\alpha\int_{\Omega_k}\varphi^2|\widetilde{u}_k|^{\alpha}|\widetilde{v}_k|^{\beta} \nonumber \\
&\qquad = \partial_{u}E(u_k,v_k)[\varphi_k^2u_k] = o(1), \label{eq:partial=0}
\end{align}
and similarly for $\widetilde{v}_k$.

In order to show that $(\widetilde{u},\widetilde{v})\neq (0,0)$ we argue by contradiction. Assume that $\widetilde{u}=0=\widetilde{v}$. Then, $\widetilde{u}_k\to 0$ and $\widetilde{v}_k\to 0$ in $L_{\mathrm{loc}}^{2}(\mathbb{R}^{N})$. Let $\varphi \in \mathcal{C}_{c}^{\infty}(B_{1}(y))$ with $y\in \mathbb{R}^{N}$. From \eqref{eq:partial=0} we get
\begin{align*}
\int_{\Omega _{k}}|\nabla(\varphi \widetilde{u}_{k})|^{2} &=\int_{\Omega _{k}}\nabla\widetilde{u}_{k}\cdot \nabla(\varphi^{2}\widetilde{u}_{k}) + \int_{\Omega _{k}}\widetilde{u}_{k}^{2}\,|\nabla \varphi|^{2} \\
&=\mu_{1}\int_{\Omega_k}\varphi^2|\widetilde{u}_k|^{2^{*}} + \lambda\alpha\int_{\Omega_k}\varphi^2|\widetilde{u}_k|^{\alpha}|\widetilde{v}_k|^{\beta} +o(1), 
\end{align*}
and a similar expression for $\widetilde{v}_k$. Adding both identities, and using Hölder's inequality and inequalities \eqref{eq:sup} and \eqref{eq:delta}, we obtain

\begin{align*}
\|(\varphi \widetilde{u}_{k},\varphi \widetilde{v}_{k})\|^{2} &\leq\int_{\mathbb{R}^N}\varphi^{2}f(\widetilde{u}_k,\widetilde{v}_k) = \int_{\mathbb{R}^N} f(\widetilde{u}_k,\widetilde{v}_k)^{\frac{2^{*}-2}{2^{*}}} f(\varphi \widetilde{u}_k,\varphi \widetilde{v}_k)^{\frac{2}{2^{*}}} \\
&\leq \left(\int_{B_{1}(y)}f(\widetilde{u}_k,\widetilde{v}_k)\right)^{\frac{2^{*}-2}{2^{*}}}\left(\int_{\mathbb{R}^N}f(\varphi \widetilde{u}_k,\varphi \widetilde{v}_k)\right)^{\frac{2}{2^{*}}}+o(1) \\
&\leq \delta^{\frac{2}{N}}\widehat{S}^{-1}\|(\varphi \widetilde{u}_{k},\varphi \widetilde{v}_{k})\|^{2} + o (1) \\
&< \left(\frac{1}{2}\right)^{\frac{2}{N}}\|(\varphi \widetilde{u}_{k},\varphi \widetilde{v}_{k})\|^{2}. 
\end{align*}
This implies that $\|(\varphi \widetilde{u}_{k},\varphi \widetilde{v}_{k})\|=o(1)$ and, hence, that $|\varphi\widetilde{u}_{k}|_{2^{*}}=o(1)$ and $|\varphi\widetilde{v}_{k}|_{2^{*}}=o(1)$ for every $\varphi \in \mathcal{C}_{c}^{\infty }(B_{1}(y))$ and every $y\in \mathbb{R}^{N}$. Therefore, $\widetilde{u}_{k}\to 0$ and $\widetilde{v}_{k}\to 0$ in $L_{\mathrm{loc}}^{2^{*}}(\mathbb{R}^{N})$, contradicting \eqref{eq:sup}. This proves that $(\widetilde{u},\widetilde{v})\neq (0,0)$.

Passing to a subsequence, we have that $\xi_k\to\xi$ and $\varepsilon_k\to\varepsilon\in[0,\infty)$. Moreover, $\varepsilon =0$, otherwise, as $u_k\rightharpoonup 0$  and $v_k\rightharpoonup 0$ weakly in $D_{0}^{1,2}(\Omega)$, we would get that $\widetilde{u}=0=\widetilde{v}$, which is a contradiction. Also, passing to a subsequence, 
\begin{equation*}
  \varepsilon_k^{-1} \mathrm{dist}(\xi_k,\partial\Omega) \to d \in [0,\infty] \quad \text{as } k \to \infty.
\end{equation*}
Next, we show that $d=\infty$.

Arguing by contradiction, assume that $d \in [0,\infty)$. Then, as $\varepsilon_k \to 0$, we have that $\xi \in \partial \Omega$. If a subsequence of $(\xi_k)$ is contained in $\bar \Omega$ we set $\bar d:=-d$, otherwise we set $\bar d:=d$. We define 
$$\mathbb{H} :=\{ y\in\mathbb{R}^N: y\cdot\nu > \bar d \},$$
where $\nu$ is the inner unit normal to $\partial \Omega$ at $\xi$. It is easy to see that, if $X$ is compact and $X \subset \mathbb{H}$, then $X\subset \Omega_k$ for $k$ large enough and, if  $X$ is compact and $X \subset \mathbb{R}^N \smallsetminus \bar{\mathbb{H}}$, then $X\subset \mathbb{R}^N \smallsetminus\Omega_k$ for $k$ large enough. As $\widetilde{u}_k \to\widetilde{u}$ and $\widetilde{v}_k \to\widetilde{v}$ a.e. in $\mathbb{R}^N$, this implies, in particular, that $\widetilde{u}=0=\widetilde{v}$ a.e. in $\mathbb{R}^N \smallsetminus \mathbb{H}$. So $(\widetilde{u},\widetilde{v})\in \mathscr{D}(\mathbb{H})$. Moreover, if $\varphi\in\mathcal{C}_c^{\infty}(\mathbb{H})$ then, as $\mathrm{supp}(\varphi)\subset\Omega_k$ for $k$ large enough, setting $\varphi_k(x):=\varepsilon_{k}^{\frac{2-N}{2}}\varphi(\frac{x-\xi_k}{\varepsilon_k})$, we have that $(\varphi_k)$ is a bounded sequence in $D^{1,2}_0(\Omega)$. Hence,
\begin{align*}
\partial_{u}E(\widetilde{u}_k,\widetilde{v}_k)\varphi&=\int_{\mathbb{H}}\nabla \widetilde{u}_k\cdot\nabla \varphi - \mu_{1}\int_{\mathbb{H}}|\widetilde{u}_k|^{2^{*}-2}\widetilde{u}_k\varphi - \lambda\alpha\int_{\mathbb{H}}|\widetilde{u}_k|^{\alpha -2}|\widetilde{v}_k|^{\beta}\widetilde{u}_k\varphi \\
&= \partial_{u}E(u_k,v_k)\varphi_k = o(1).
\end{align*}
Similarly, $\partial_{v}E(\widetilde{u}_k,\widetilde{v}_k)\varphi = o(1)$. Passing to the limit as $k\to\infty$, we conclude that $(\widetilde{u},\widetilde{v})$ solves the system \eqref{eq:system} in $\mathbb{H}$, contradicting Proposition \ref{prop:pohozhaev}.

This proves that $d=\infty$ and, by \eqref{eq:xi}, we have that $\xi_k \in\Omega$. Moreover, every compact subset $X$ of $\mathbb{R}^N$ is contained in $\Omega_k$ for $k$ large enough. So, arguing as above, we conclude that $(\widetilde{u},\widetilde{v})$ solves the limit system \eqref{eq:limitsystem} in $\mathbb{R}^N$.

Let $G/K=\{[g_1],...,[g_n]\}$. Set
$$r_k:=\frac{1}{4}\min\{\mathrm{dist}(\xi_k,\partial\Omega),\,|g_{i}(\xi_k)-g_{j}(\xi_k)|:i,j=1,...,n,\,i\neq j\}.$$
Choose a radially symmetric cut-off function $\chi\in\mathcal{C}_{c}^{\infty}(\mathbb{R}^{N})$ such that $0\leq\chi\leq1$, $\chi(x)=1$ if $|x| \leq1$ and $\chi(x)=0$ if $|x| \geq2$, and define
\begin{align*}
w_k(x)&:=u_k(x)-\sum_{i=1}^{n}\varepsilon_k^{\frac{2-N}{2}}\widetilde{u}\left(g_{i}^{-1}\left(\frac{x-g_{i}\xi_k}{\varepsilon_k}\right)  \right)\chi(r_k^{-1}(x-g_i\xi_k)),  \\
z_k(x)&:=v_k(x)-\sum_{i=1}^{n}\varepsilon_k^{\frac{2-N}{2}}\widetilde{v}\left(g_{i}^{-1}\left(\frac{x-g_{i}\xi_k}{\varepsilon_k}\right)  \right)\chi(r_k^{-1}(x-g_i\xi_k)).
\end{align*}
Since $\widetilde{u}$ and $\widetilde{v}$ are $K$-invariant and $G_{\xi_k}=K$ for all
$k\in\mathbb{N}$, we have that $w_k$ and $z_k$ are $G$-invariant. Clearly, $(w_k,z_k)\rightharpoonup (0,0)$ weakly in $\mathscr{D}(\Omega)$. Now, for each $j=1,...,n$, we define
\begin{align*}
w_k^{j}(x)&:=u_k(x)-\sum_{i=j}^{n}\varepsilon_k^{\frac{2-N}{2}}\widetilde{u}\left(g_{i}^{-1}\left(\frac{x-g_{i}\xi_k}{\varepsilon_k}\right)  \right)  \\
z_k^{j}(x)&:=v_k(x)-\sum_{i=j}^{n}\varepsilon_k^{\frac{2-N}{2}}\widetilde{v}\left(g_{i}^{-1}\left(\frac{x-g_{i}\xi_k}{\varepsilon_k}\right)  \right).
\end{align*}
As $r_k\varepsilon_k^{-1}\to\infty$, an easy computation shows that
\begin{equation} \label{eq:cutoff}
\left\|u_k-w_k-\sum_{i=1}^{n}\varepsilon_k^{\frac{2-N}{2}}\widetilde{u}\left(g_{i}^{-1}\left(\frac{\cdot-g_{i}\xi_k}{\varepsilon_k}\right)  \right)  \right\| =\left\|w_k^{1}-w_k\right\| \to 0,
\end{equation}
and similarly for $v_k$. This proves that $(w_k,z_k)$ satisfies \emph{(iii)}.

We rescale $w_k^{j}$ and use the $G$-invariance of $u_k$ to obtain
\begin{align*}
&\widetilde{w}_k^{j}(y):=\varepsilon_k^{\frac{N-2}{2}}w_k^{j}(\varepsilon_ky+g_{j}\xi_k)\\
&\quad =\varepsilon_k^{\frac{N-2}{2}}u_k(\varepsilon_ky+g_{j}\xi_k)-\sum_{i=j+1}^{n}\widetilde{u}\left(g_{i}^{-1}\left(\frac{\varepsilon_k y+ g_{j}\xi_k-g_{i}\xi_k}{\varepsilon_k}\right)  \right)  -\widetilde{u}(g_{j}^{-1}y)\\
&\quad =\widetilde{u}_k(g_{j}^{-1}y)-\sum_{i=j+1}^{n}\widetilde{u}\left(g_{i}^{-1}\left(y+\frac{g_{j}\xi_k-g_{i}\xi_k}{\varepsilon_k}\right)\right) - \widetilde{u}(g_{j}^{-1}y).
\end{align*}
Since $\widetilde{u}_k\rightharpoonup\widetilde{u}$ weakly in $D^{1,2}(\mathbb{R}^{N})$ and $\varepsilon_k^{-1}|g_{j}\xi_k-g_{i}\xi_k|\to\infty$ for every $i\neq j$, we have that
\begin{equation*}
\widetilde{u}_k\circ g_{j}^{-1}-\sum_{i=j+1}^{n}\widetilde{u}\left(g_{i}^{-1}\left(\cdot +\frac{g_{j}\xi_k-g_{i}\xi_k}{\varepsilon_k}\right)\right)  \rightharpoonup\widetilde{u}\circ g_{j}^{-1}\quad\text{weakly in }D^{1,2}(\mathbb{R}^{N}).
\end{equation*}
Therefore,
\begin{align}
\|\widetilde{w}_k^j\|^2 &= \left\|\widetilde{u}_k\circ g_{j}^{-1}-\sum_{i=j+1}^{n}(\widetilde{u}\circ g_{i}^{-1})\left(\cdot +\frac{g_{j}\xi_k-g_{i}\xi_k}{\varepsilon_k}\right)\right\|^2  - \|\widetilde{u}\circ g_{j}^{-1}\|^2 + o(1) \nonumber \\
&=\|\widetilde{w}_k^{j+1}\|^2 - \|\widetilde{u}\|^2 + o(1), \label{eq:iteration}
\end{align}
where the second equality is given by the change of variable $x=\varepsilon_k y + g_j\xi_k$. Iterating the identity \eqref{eq:iteration} and using \eqref{eq:cutoff} we obtain
\begin{equation} \label{eq:E0}
\|u_k\|^2 = \|w_k\|^2 + n\|\widetilde{u}\|^2 + o(1).
\end{equation}
Adding this last identity with the similar one for $v_k$ gives
\begin{equation} \label{eq:E1}
\|(u_k,v_k)\|^2 = \|(w_k,z_k)\|^2 + n\|(\widetilde{u},\widetilde{v})\|^2 + o(1).
\end{equation}
A similar argument, using Lemma \ref{lem:bl} yields
\begin{align}
\int_{\Omega}\left(\mu_1|u_k|^{2^*}+\mu_2|v_k|^{2^*}\right) = &\int_{\Omega}\left(\mu_1|w_k|^{2^*}+\mu_2|z_k|^{2^*}\right) \nonumber \\
&+ n\int_{\mathbb{R}^N}\left(\mu_1|\widetilde{u}|^{2^*}+\mu_2|\widetilde{v}|^{2^*}\right) + o(1), \label{eq:E2}
\end{align}
and
\begin{equation} \label{eq:E3}
\int_{\Omega} |u_k|^{\alpha}|v_k|^{\beta} = \int_{\Omega} |w_k|^{\alpha}|z_k|^{\beta} + n\int_{\mathbb{R}^N} |\widetilde{u}|^{\alpha}|\widetilde{v}|^{\beta}+o(1).
\end{equation}
From \eqref{eq:E1}, \eqref{eq:E2} and \eqref{eq:E3} we obtain
$$E(u_k,v_k) = E(w_k,z_k) + nE(\widetilde{u},\widetilde{v}) + o(1).$$
This proves \emph{(iv)}.

In a similar way, using Lemma \ref{lem:derivative}, we get that
$$o(1) = E'(u_k,v_k) = E'(w_k,z_k) + nE'(\widetilde{u},\widetilde{v}) + o(1) = E'(w_k,z_k) + o(1)$$
in $(\mathscr{D}(\mathbb{R}^N))'$. This completes the proof of \emph{(v)}. 

The proof of Lemma \ref{lem:struwe} is now complete.
\end{proof}

\begin{proof}[Proof of Theorem \ref{thm:struwe}]
The sequence $((u_k,v_k))$ is bounded in $\mathscr{D}(\Omega)$. So, after passing to a subsequence, we have that $u_k\rightharpoonup u$ and $v_k\rightharpoonup v$ weakly in $D^{1,2}(\Omega)$, $u_k\to u$ and $v_k\to v$ in $L_{\mathrm{loc}}^{2}(\Omega)$, and $u_k\to u$ and $v_k\to v$ a.e. in $\Omega$. A standard argument shows that $(u,v)$ solves the system \eqref{eq:system}. If $(u_k,v_k)\to(u,v)$ strongly in $\mathscr{D}(\Omega)$, the proof is finished. If not, we set $u_k^1:=u_k-u$ and $v_k^1:=v_k-v$. From Lemmas \ref{lem:bl} and \ref{lem:derivative} we obtain that
\begin{align*}
E(u_k,v_k) &= E(u_k^1,v_k^1) + E(u,v) + o(1), \\
E'(u_k,v_k) &= E'(u_k^1,v_k^1) + E'(u,v) + o(1) = E'(u_k^1,v_k^1)+ o(1)\quad\text{in }(\mathscr{D}(\Omega))'.
\end{align*}
As $(u_k^1,v_k^1)\not\to(0,0)$, Lemma \ref{lem:struwe} yields a closed subgroup $K_1$ of finite index in $G$, sequences $(\xi_{1,k})$ in $\Omega$ and $(\varepsilon_{1,k})$ in $(0,\infty)$, a nontrivial $K_1$-invariant solution $(\widetilde{u}_1, \widetilde{v}_1)$ to the limit system \eqref{eq:limitsystem}, and a sequence $((w_{1,k},z_{1,k}))$ which satisfy the statements \emph{(i)-(v)} of Lemma \ref{lem:struwe}. If $(w_{1,k},z_{1,k})\to (0,0)$, the proof is finished. If not, we apply the Lemma \ref{lem:struwe} again to the sequence $((u_k^2,v_k^2))$ with $(u_k^2,v_k^2):=(w_{1,k},z_{1,k})$. We continue this way until, after a finite number of steps, we reach a sequence $((u_k^m,v_k^m))$ that converges strongly to $(0,0)$ in $\mathscr{D}(\Omega)$. This completes the proof.
\end{proof}

\section{Cooperative systems} \label{sec:cooperative}

This section is devoted to the proof of Theorems \ref{thm:bc}, \ref{thm:bounded} and \ref{thm:entire} in the cooperative case.

\begin{lemma} \label{lem:synchronized}
Let $w$ be a nontrivial solution to the problem \eqref{eq:bc}. Then, there exist $s,t>0$ such that $(sw,tw)$ is a solution to the system \eqref{eq:system} if and only if there exists $r>0$ such that
\begin{equation} \label{eq:h}
h(r):=\mu_1 r^{2^*-2} + \lambda\alpha r^{\alpha -2}-\lambda\beta r^\alpha -\mu_2=0 \quad\text{and}\quad \mu_2+\lambda\beta r^\alpha>0.
\end{equation}
In particular, if $\lambda>0$ and assumption $(A)$ is satisfied, then there exist $s,t>0$ such that $(sw,tw)$ solves \eqref{eq:system}.
\end{lemma}

\begin{proof}
Let $w$ be a nontrivial solution to the problem \eqref{eq:bc}. Then, $(sw,tw)$ solves the system \eqref{eq:system} if and only if $(s,t)$ solves the system
\begin{equation} \label{eq:num_sys}
\begin{cases}
\mu_1 s^{2^*-2}+\lambda\alpha s^{\alpha-2}t^\beta = 1 \\
\mu_2 t^{2^*-2}+\lambda\beta s^\alpha t^{\beta -2} = 1.
\end{cases}
\end{equation}
It is easy to see that, if $r>0$ satisfies $h(r)=0$ and $\mu_2+\lambda\beta r^\alpha>0$, and if we set $t:= (\mu_2 + \lambda\beta r^{\alpha})^{\frac{-1}{2^*-2}}$, then $(rt,t)$ satisfies \eqref{eq:num_sys}. Conversely, if $(s,t)$ solves the system \eqref{eq:num_sys} and $s,t>0$, setting $r:=\frac{s}{t}$ we get that $h(r)=0$ and $\mu_2+\lambda\beta r^\alpha>0$. 

Let $\lambda>0$. If $\alpha\in (2^*-2,2)$, then 
$$\lim_{r\to 0^+}h(r)=\infty \qquad \text{ and } \qquad \lim_{r\to\infty}h(r) = -\infty.$$
Hence, there exists $r>0$ such that $h(r)=0$, and \eqref{eq:num_sys} holds true. As $\alpha\in (1, 2^*-1)$, we have that $\alpha\in (2^*-2,2)$ if $N\geq 6$. If $N\leq5$ and $\alpha\not\in (2^*-2,2)$, assumption $(A)$ yields the existence of $r>0$ with $h(r)=0$. This completes the proof.
\end{proof}

So, for $\lambda>0$, Theorems \ref{thm:bc}, \ref{thm:bounded} and \ref{thm:entire} follow immediately from the corresponding statements for the problem \eqref{eq:bc}.

\begin{proof}[Proof of Theorem \ref{thm:bc}]
Bahri and Coron showed in \cite[Theorem 1]{bc} that problem \eqref{eq:bc} has a positive solution if $\widetilde{H}_*(\Omega;\mathbb{Z}_2)\neq 0$. Now Lemma \ref{lem:synchronized} yields the result for the system \eqref{eq:system} with $\lambda>0$.
\end{proof}

\begin{proof}[Proof of Theorem \ref{thm:bounded} for $\lambda>0$.]
Let $\Gamma$ be the closed subgroup of $O(N)$ and $\Theta$ be the $\Gamma$-invariant bounded smooth domain in $\mathbb{R}^N$ given in the statement of Theorem \ref{thm:bounded}. It was shown in \cite[Theorem 1]{cf0} that, for any given $n\in\mathbb{N}$, there exists $\ell_n >0$, depending on $\Gamma$ and $\Theta$, such that, if $G$ is closed subgroup of $\Gamma$, $\Omega$ is a $G$-invariant bounded smooth domain in $\mathbb{R}^N$ which contains $\Theta$, and
$$\min_{x\in\bar{\Omega}} \#Gx > \ell_n,$$
then the problem \eqref{eq:bc} has at least $n$ pairs of nontrivial $G$-invariant solutions $\pm w_1,...,\pm w_n$ such that $w_1$ is positive and $w_2,...,w_n$ change sign. This result, together with Lemma \ref{lem:synchronized}, yields Theorem \ref{thm:bounded} for $\lambda>0$.
\end{proof}

\begin{proof}[Proof of Theorem \ref{thm:entire} for $\lambda>0$.]
W.Y. Ding showed in \cite{d} that the problem 
$$-\Delta w=|w|^{{2}^{*}-2}w,\qquad w\in D^{1,2}(\mathbb{R}^N),$$
has a sequence $(w_k)$ of solutions such that $\|w_k\|\to\infty$ as $k\to\infty$; see also \cite{c,dmpp}. 
This result, together with Lemma \ref{lem:synchronized}, yields Theorem \ref{thm:entire} for $\lambda>0$.
\end{proof}

\section{Competitive systems} \label{sec:competitive}

This section is devoted to the proof of Theorems \ref{thm:bounded} and \ref{thm:entire} in the competitive case. We will assume throughout that $\lambda<0$.

It was shown in \cite[Proposition 2.3]{cp} that there exists $\lambda_* <0$ such that the system \eqref{eq:system} does not have a fully nontrivial synchronized solution if $\lambda<\lambda_*$. So the argument given in the previous section for the cooperative case does not work in the competitive case. We will use a $\mathcal{C}^1$-Ljusternik-Schnirelmann theorem due to A. Szulkin, stated below.

Let $X$ be a real Banach space, $M$ be a closed $\mathcal{C}^1$-submanifold of $X$ and $\Phi\in\mathcal{C}^1(M,\mathbb{R})$. We denote by $d_x\Phi$ the differential of $\Phi$ at a point $x\in M$, and write
$$K_c: = \{x\in M:\Phi(x)=c\text{ and }d_x\Phi=0\}.$$
Recall that $\Phi$ is said to satisfy the $(PS)_c$-\emph{condition on} $M$ if every sequence $(x_k)$ in $M$ such that $\Phi(x_k)\to c$ and $\|d_x\Phi\|\to 0$ has a convergent subsequence.

Let $Z$ be a symmetric subset of $X$ with $0\not\in Z$. Recall that $Z$ is called \emph{symmetric} if $-Z=Z$. If $Z$ is nonempty, the \emph{genus of} $Z$ is the smallest integer $j\geq 1$ such that there exists an odd continuous function $Z\to\mathbb{S}^{j-1}$ into the unit sphere $\mathbb{S}^{j-1}$ in $\mathbb{R}^j$. We denote it by $\mathrm{genus}(Z)$. If no such $j$ exists, we set $\mathrm{genus}(Z):=\infty$. We define $\mathrm{genus}(\emptyset):=0$. The properties of the genus may be found in \cite[Proposition 2.3]{s}.

If $M$ is symmetric and $0\not\in M$, we define
$$c_j:=\inf_{Z\in\Sigma_j}\max_{x\in Z}\Phi(x),$$
where $\Sigma_j:=\{Z\subset M:Z\text{ is symmetric and compact, and }\mathrm{genus}(Z)\geq j\}$.

A closer look at the proof of \cite[Theorem 3.1]{s} allows to give a more detailed statement of \cite[Corollary 4.1]{s}, as follows.

\begin{theorem}[Szulkin \cite{s}] \label{thm:szulkin}
Let $M$ be a closed symmetric $\mathcal{C}^1$-submanifold of $X$ which does not contain the origin, and let $\Phi\in\mathcal{C}^1(M,\mathbb{R})$ be an even function which is bounded below. If $\Sigma_n\neq\emptyset$ for some $n\geq 1$ and $\Phi$ satisfies $(PS)_{c_j}$ for $j=1,...,n$, then, for each $j=1,...,n$, one has that $c_j$ is a critical value of $\Phi$ and
$$\mathrm{genus}(K_{c_j})\geq m+1 \qquad \text{if }\;c_j=\cdots=c_{j+m}\;\text{ for some }m\geq 0.$$
\end{theorem}

We shall use this theorem to prove Theorems \ref{thm:bounded} and \ref{thm:entire} for $\lambda<0$.

\subsection{Multiplicity in bounded domains}

Let $G$ be a closed subgroup of $O(N)$ and let $\Omega$ be a $G$-invariant bounded smooth domain in $\mathbb{R}^N$. Set
$$\mathcal{N}(\Omega)^G:=\mathcal{N}(\Omega) \cap \mathscr{D}(\Omega)^G.$$
It is easy to see that $\nabla E(u,v), \,\nabla F_1(u,v), \,\nabla F_2(u,v)\in\mathscr{D}(\Omega)^G$ for every $(u,v)\in\mathscr{D}(\Omega)^G$, where $F_1$ and $F_2$ are the functions defined in Proposition \ref{prop:nehari}; see, e.g., \cite[Theorem 1.28]{w}. Therefore, $\mathcal{N}(\Omega)^G$ is a closed $\mathcal{C}^1$-submanifold of $\mathscr{D}(\Omega)^G$ and a natural constraint for $E$.

The orthogonal projection of $\nabla E(u,v)$ onto the tangent space of $\mathcal{N}(\Omega)^G$ at the point $(u,v)$ will be denoted by $\nabla_{\mathcal{N}(\Omega)}E(u,v)$. 

\begin{lemma} \label{lem:PSonN}
If $((u_{k},v_{k}))$ is a sequence in $\mathcal{N}(\Omega)^G$ such that
$$E(u_{k},v_{k})\to c \qquad \text{and} \qquad\nabla_{\mathcal{N}(\Omega)}E(u_{k},v_{k})\to 0,$$
then $\nabla E(u_{k},v_{k})\to 0$.
\end{lemma}

\begin{proof}
The same argument used in \cite[Lemma 3.5]{cp} yields this statement.
\end{proof}

Lemma \ref{lem:PSonN} allows us to apply Theorem \ref{thm:struwe} to obtain the following compactness condition.

\begin{lemma} \label{lem:PS}
If $\min_{x\in\overline{\Omega}}\#Gx\geq2$, then the functional $E$ satisfies the $(PS)_{c}$-condition on $\mathcal{N}(\Omega)^G$ for every 
$$c\leq\left(1+\min_{x\in\overline{\Omega}}\#Gx\right)\frac{1}{N}\mu_{0}^\frac{2-N}{2}S^\frac{N}{2},$$
where $\mu_{0}:=\max\{\mu_{1},\mu_{2}\}$. In particular, if $\#Gx=\infty$ for every $x\in\overline{\Omega}$, then $E$ satisfies the $(PS)_{c}$-condition on $\mathcal{N}(\Omega)^G$ for every $c\in\mathbb{R}$.
\end{lemma}

\begin{proof}
Let $((u_k,v_k))$ be a sequence such that
$$(u_k,v_k)\in \mathcal{N}(\Omega)^{G},\qquad E(u_k,v_k)\to c,\qquad \nabla_{\mathcal{N}(\Omega)}E(u_k,v_k)\to 0.$$
By Lemma \ref{lem:PSonN}, we have that $\nabla E(u_{k},v_{k})\to 0$.

Arguing by contradiction, assume that the number $m$ given by Theorem \ref{thm:struwe} is such that $m\geq 1$, and let $(\widetilde{u}_1,\widetilde{v}_1),...,(\widetilde{u}_m,\widetilde{v}_m)$ be the nontrivial solutions to the limit problem \eqref{eq:limitsystem} given by that theorem.

If $\widetilde{u}_j \neq 0$ and $\widetilde{v}_j \neq 0$ for some $j=1,...,m$, then, by Proposition \ref{prop:nonmin}, 
$$E(\widetilde{u}_j, \widetilde{v}_j) > \frac{1}{N}(\mu_{1}^\frac{2-N}{2}+\mu_{2}^\frac{2-N}{2})S^\frac{N}{2}.$$
If, on the other hand, $\widetilde{u}_j = 0$ and $\widetilde{v}_i = 0$ for some $i,j=1,...,m$, then, as $(\widetilde{u}_i,\widetilde{v}_i)\neq(0,0)\neq (\widetilde{u}_j,\widetilde{v}_j)$, we have that $i\neq j$ and
$$E(\widetilde{u}_i,\widetilde{v}_i) + E(\widetilde{u}_j, \widetilde{v}_j) \geq \frac{1}{N}(\mu_{1}^\frac{2-N}{2}+\mu_{2}^\frac{2-N}{2})S^\frac{N}{2}.$$
In both cases Theorem \ref{thm:struwe} yields 
$$c \geq 2\left(\min_{x\in\overline{\Omega}}\#Gx \right) \frac{1}{N}\mu_{0}^\frac{2-N}{2}S^\frac{N}{2},$$
contradicting our assumption. Finally, if $\widetilde{u}_1 = \cdots = \widetilde{u}_m = 0$, then, by Theorem \ref{thm:struwe}, $u_k\to u$ strongly in $D^{1,2}_0(\Omega)$, where $u$ is the first component of a solution $(u,v)$ to \eqref{eq:system}, and $\widetilde{v}_j \neq 0$ for all $j=1,...,m$. Hence,
\begin{align*}
c &\geq \frac{1}{N}(\|u\|^{2}+\|v\|^{2}) + m\left(\min_{x\in\overline{\Omega}}\#Gx \right) \frac{1}{N}\mu_{2}^\frac{2-N}{2}S^\frac{N}{2} \\
&> \left(1+\min_{x\in\overline{\Omega}}\#Gx \right) \frac{1}{N}\mu_{0}^\frac{2-N}{2}S^\frac{N}{2},
\end{align*}
which is, again, a contradiction. 

Therefore, $m=0$. This completes the proof.
\end{proof}

\begin{lemma} \label{lem:genus}
For every $j\geq 1$, the set
$$\Sigma_j^G(\Omega):=\{Z\subset \mathcal{N}(\Omega)^{G}:Z\text{ is symmetric and compact, and }\mathrm{genus}(Z)\geq j\}$$
is nonempty.
\end{lemma}

\begin{proof}
Fix $j\geq 1$ and choose $2j$ pairwise disjoint open $G$-invariant subsets of $\Omega$ and nontrivial $G$-invariant functions $u_i\in\mathcal{C}_c^\infty(U_i)$ and $v_i\in\mathcal{C}_c^\infty(U_{j+i})$, $i=1,...,j$.

For $(u,v)\in \mathscr{D}(\Omega)$ with $u\neq 0$, $v\neq 0$, let $s_u$ and $t_v$ be the unique positive numbers such that $\|s_u u\|^2=\int_{\Omega}\mu_1 |s_u u|^{2^*}$ and $\|t_v v\|^2=\int_{\Omega}\mu_2 |t_v v|^{2^*}$, and set
$$\varrho(u,v):=(s_u u, t_v v).$$
Note that $\varrho(-u,-v)=-\varrho(u,v)$ and that $\varrho(u,v)\in\mathcal{N}(\Omega)$ if $uv=0$.

Let $\{e_1,...,e_j\}$ be the canonical basis of $\mathbb{R}^j$. The boundary of the convex hull of the set $\{\pm e_1,...,\pm e_j\}$, which is given by
$$Q:=\left\{\sum_{i=1}^j r_i\tilde{e}_i:\;\tilde{e}_i\in\{e_i,-e_i\},\;r_i\in [0,1],\;\sum_{i=1}^j r_i=1\right\},$$
has genus $j$ and the map $\psi:Q\to\mathcal{N}(\Omega)^{G}$, given by
$$\psi(e_i):=\varrho(u_i,v_i),\quad\psi(-e_i):=-\psi(e_i),\quad \psi\left(\sum_{i=1}^j r_i\tilde{e}_i\right):=\varrho \left(\sum_{i=1}^j r_i\psi(\tilde{e}_i)\right),$$
is odd and continuous. Hence, the set $Z:=\psi(Q)\subset\mathcal{N}(\Omega)^{G}$ is symmetric and compact, and $\mathrm{genus}(Z)\geq j$; see \cite[Proposition 2.3]{s}. This completes the proof.
\end{proof}

Set
$$c_j^G(\Omega):=\inf_{Z\in\Sigma_j^G(\Omega)}\,\max_{(u,v)\in Z}E(u,v).$$

\begin{proof}[Proof of Theorem \ref{thm:bounded} for $\lambda<0$.]
Let $\Gamma$ be the subgroup of $O(N)$ and $\Theta$ be the $\Gamma$-invariant bounded smooth domain given in the statement of Theorem \ref{thm:bounded}. Lemma \ref{lem:genus} asserts that $\Sigma_j^G(\Theta)\neq\emptyset$. Hence, $c_j^\Gamma(\Theta)\in\mathbb{R}$ for every $j\in \mathbb{N}$.

Given $n\in \mathbb{N}$, we define
$$\ell_n:= \max\left\lbrace 1,\; c_n^\Gamma(\Theta)\left(N\mu_0^{\frac{N-2}{2}}S^{-\frac{N}{2}}\right)-1\right\rbrace,$$
where $\mu_0:=\max\{\mu_1,\mu_2\}$.

If $G$ is a closed subgroup of $\Gamma$ and $\Omega$ is a $G$-invariant bounded smooth domain in $\mathbb{R}^N$ which contains $\Theta$, then $\mathcal{N}(\Theta)^{\Gamma}\subset\mathcal{N}(\Omega)^{G}$ and, hence, $\Sigma_n^\Gamma(\Theta)\subset\Sigma_n^G(\Omega)$ and $c_n^G(\Omega)\leq c_n^\Gamma(\Theta)$. So, if
$$\min_{x\in\bar{\Omega}} \#Gx > \ell_n,$$
then, for each $1\leq j\leq n$, we have that
$$c_j:=c_j^G(\Omega)\leq c_n^G(\Omega)\leq c_n^\Gamma(\Theta) < \left(1+\min_{x\in\bar{\Omega}} \#Gx\right)\frac{1}{N}\mu_0^{\frac{2-N}{2}}S^{\frac{N}{2}}.$$

It follows from Lemma \ref{lem:PS} that $E$ satisfies the $(PS)_{c_j}$-condition on $\mathcal{N}(\Omega)^G$ for every $1\leq j\leq n$. Moreover, by Lemma \ref{lem:genus}, $\Sigma_n^G(\Omega)\neq\emptyset$. Therefore, by Theorem \ref{thm:szulkin}, $E$ has $n$ critical points $(u_1,v_1),...,(u_n,v_n)$ on $\mathcal{N}(\Omega)^{G}$ with $E(u_j,v_j)=c_j^G(\Omega)$. 

If $c_i^G(\Omega)\neq c_j^G(\Omega)$ then $\|(u_i,v_i)\|\neq\|(u_j,v_j)\|$ and these points are nonequivalent in the sense defined in the introduction. If $c_i^G(\Omega)=\cdots =c_{i+m}^G(\Omega)=:c$ for some $m\geq1$, then $\mathrm{genus}(K_{c})\geq2$. This implies that $\#K_c=\infty$, so $E$ has infinitely many nonequivalent critical points with critical value $c$.  

As $c_1^G(\Omega)=\inf_{(u,v)\in\mathcal{N}(\Omega)^G}E(u,v)$ and $E(|u|,|v|)=E(u,v)$, after replacing $(u_1,v_1)$ with $(|u_1|,|v_1|)$ we get a positive critical point. The proof is complete.
\end{proof}

\subsection{Phase separation in bounded domains}

In this subsection we prove Theorem \ref{thm:separation}. 

Let $G$ be a closed subgroup of $O(N)$ and  $\Omega$ be a $G$-invariant smooth bounded domain. Consider the problem
\begin{equation}\label{eq:limitprob}
-\Delta w=\mu_{1}|w^{+}|^{2^{*}-2}w^{+}+\mu_{2}|w^{-}|^{2^{*}-2}w^{-},\qquad w\in D^{1,2}_0(\Omega)^{G},
\end{equation}
where $w^{+}=\max\{w,0\}$ and $w^{-}=\min\{w,0\}$. Let 
\begin{equation*}
J(w):=\frac{1}{2}\int_{\Omega}|\nabla w|^{2}-\frac{1}{2^{*}}\int_{\Omega}(\mu_{1}|w^{+}|^{2^{*}}+\mu_{2}|w^{-}|^{2^{*}}),
\end{equation*}
be its energy functional and  
\begin{align*}
\mathcal{M}^{G}:&=\{w\in D^{1,2}_0(\Omega)^{G}: w\neq 0, J'(w)w=0 \} \\
&=\left\lbrace w\in D^{1,2}_0(\Omega)^{G}: w\neq 0, \int_{\Omega}|\nabla w|^{2}=\int_{\Omega}(\mu_{1}|w^{+}|^{2^{*}}+\mu_{2}|w^{-}|^{2^{*}}) \right\rbrace
\end{align*}
its Nehari manifold. The sign-changing solutions to problem \eqref{eq:limitprob}  lie on the set
\begin{equation*}
\mathcal{E}^{G}:=\{w\in D^{1,2}_0(\Omega)^{G}: w^{+}\in \mathcal{M}^{G},\; w^{-}\in \mathcal{M}^{G}  \}.
\end{equation*}
It is easy to see that $\mathcal{E}^{G}\neq \emptyset$. We define 
\begin{equation*}
c_{\infty}^{G}:=\inf_{w\in \mathcal{E}^{G}}J(w)<\infty. 
\end{equation*}

To emphasize the dependence on $\lambda$, in the following we write $\mathcal{N}_{\lambda}(\Omega)^{G}$ and $E_{\lambda}$, instead of $\mathcal{N}(\Omega)^{G}$ and $E$, for the Nehari manifold and the energy funcional of the system \eqref{eq:system}. Notice that $(w^{+},w^{-})\in \mathcal{N}_{\lambda}^{G}$ and $J(w)=E_{\lambda}(w^{+}, w^{-})$ for every $w\in\mathcal{E}^{G}$ and each  $\lambda<0$. Therefore,  
\begin{equation}\label{eq:upper bound}
c_{\lambda}^{G}:=\inf_{(u,v)\in \mathcal{N}_{\lambda}^{G}}E_{\lambda}(u,v)\leq c_{\infty}^{G}\qquad\text{ for each } \lambda<0.
\end{equation} 

\begin{proof}[Proof of Theorem \ref{thm:separation}.]
Let $\lambda_{k}<0$, $\lambda_{k}\to -\infty$, and $(u_{k},v_{k})\in\mathcal{N}_{\lambda_k}(\Omega)^G$ be such that $u_k\geq0$, $v_k\geq0$ and, for each $k\in \mathbb{N}$,
$$c^G_{\lambda_k} = E_{\lambda_{k}}(u_{k},v_{k})=\frac{1}{N}\int_{\Omega}(|\nabla u_k|^2+|\nabla v_k|^2) \leq \left(1+\min_{x\in\bar{\Omega}} \#Gx\right)\frac{1}{N}\mu_0^{\frac{2-N}{2}}S^{\frac{N}{2}}.$$
Hence, passing to a subsequence, there exist $u_{\infty}, v_{\infty}\in D^{1,2}_0(\Omega)^{G}$ such that
\begin{align*}
u_{k}\rightharpoonup u_{\infty},\qquad &v_{k}\rightharpoonup v_{\infty}, \qquad \text{ weakly in } D^{1,2}_0(\Omega)^{G} \\
u_{k}\rightarrow u_{\infty},\qquad &v_{k}\rightarrow v_{\infty}, \qquad \text{ a.e. in } \Omega.
\end{align*}
In particular, $u_{\infty}\geq 0$ and $v_{\infty}\geq 0$. Also, since $\partial_{u}E_{\lambda_{k}}(u_{k},v_{k})u_{k}=0$, we have that
\begin{equation*}
0\leq -\alpha \lambda_{k}\int_{\Omega}|u_{k}|^{\alpha}|v_{k}|^{\beta}\leq \mu_{1} \int_{\Omega}|u_{k}|^{2^{*}}\leq C
\end{equation*}
and, from Fatou's lemma, we obtain
\begin{equation*}
\int_{\Omega}|u_{\infty}|^{\alpha}|v_{\infty}|^{\beta}\leq\liminf_{k\rightarrow\infty} \int_{\Omega}|u_{k}|^{\alpha}|v_{k}|^{\beta}\leq \frac{C}{\alpha}\lim_{k\rightarrow\infty}\frac{1}{-\lambda_{k}}=0.
\end{equation*}
Therefore, $u_{\infty}v_{\infty}=0$. We claim that
\begin{equation} \label{eq:claim}
u_{\infty}\neq0 \qquad\text{ and } \qquad v_{\infty}\neq0.
\end{equation}

To prove this claim, we argue by contradiction. Assume that $u_{\infty}=0$ and $v_{\infty}\neq 0$; the other cases can be treated in a similar way. Then, $u_{k}\rightharpoonup 0$ in $D^{1,2}_0(\Omega)$ and $\|u_{k}\|^{2}\geq c_0>0$ by Proposition \ref{prop:nehari}(a). Following the argument in the first part of the proof of Lemma \ref{lem:struwe}, one shows that there exists a closed subgroup $K$ of finite index in $G$ and sequences $(\xi_k)$ in $\mathbb{R}^N$ and $(\varepsilon_k)$ in $(0,\infty)$ such that $G_{\xi_k}=K$, \;$\varepsilon_k^{-1}|g\xi_k-\tilde{g}\xi_k|\to\infty$ for any $g,\tilde{g}\in G$ with $g^{-1}\tilde{g}\notin K$, and $\mathrm{dist}(\xi_k,\bar{\Omega})\to 0$. Moreover, the rescaled functions 
$$\widetilde{u}_k(y):=\varepsilon_k^{(N-2)/2}u_k(\varepsilon_ky+\xi_k)$$ 
converge weakly in $D^{1,2}(\mathbb{R}^N)$ to a function $\widetilde{u}$ such that $\widetilde{u}\geq 0$ and $\widetilde{u}\neq 0$. Let $G/K=\{[g_1],\ldots,[g_n]\}$ and set
$$w_k(x):=u_k(x)-\sum_{i=1}^{n}\varepsilon_k^{\frac{2-N}{2}}\widetilde{u}\left(g_{i}^{-1}\left(\frac{x-g_{i}\xi_k}{\varepsilon_k}\right)\right).$$
Arguing as we did to prove equation \eqref{eq:E0}, we obtain
\begin{equation*}
\|u_k\|^2 + o(1) = \|w_k\|^2 + n\|\widetilde{u}\|^2\geq (\min_{x\in\bar{\Omega}} \#Gx)\|\widetilde{u}\|^2.
\end{equation*}
The last inequality holds true because $\Omega$ is smooth and $\mathrm{dist}(\xi_k,\bar{\Omega})\to 0$. Define
\begin{equation*}
\widetilde{s}^{2^{*}-2}:=\frac{\int_{\mathbb{R}^n}|\nabla\widetilde{u}|^{2}}{\int_{\mathbb{R}^{N}}\mu_{1}|\widetilde{u}|^{2^{*}}} \qquad\mbox{ and }\qquad \widetilde{t}^{2^{*}-2}:=\frac{\int_\Omega|\nabla v_{\infty}|^{2}}{\int_{\Omega}\mu_{2}|v_{\infty}|^{2^{*}}}.
\end{equation*}
Set $\widetilde{v}_k(y):=\varepsilon_k^{(N-2)/2}v_k(\varepsilon_ky+\xi_k)$ and $\psi_{k}:=\varepsilon_{k}^{(2-N)/2}\widetilde{u}(\frac{x-\xi_k}{\varepsilon_k})$. Then, as $(u_k,v_k)$ solves \eqref{eq:system}, $\widetilde{u}_k\geq 0$ and $\widetilde{u}\geq 0$, we have that
\begin{align*}
0&=\partial_{u}E(u_k,v_k)\psi_{k}=\partial_{u}E(\widetilde{u}_k,\widetilde{v}_k)\widetilde{u}\\
&=\int_{\mathbb{R}^{N}}\nabla \widetilde{u}_k\cdot\nabla \widetilde{u} - \mu_{1}\int_{\mathbb{R}^{N}}|\widetilde{u}_k|^{2^{*}-1}\widetilde{u} - \lambda_k\alpha\int_{\mathbb{R}^{N}}|\widetilde{u}_k|^{\alpha -1}|\widetilde{v}_k|^{\beta}\widetilde{u} \\
&\geq \int_{\mathbb{R}^{N}}\nabla \widetilde{u}_k\cdot\nabla \widetilde{u} - \mu_{1}\int_{\mathbb{R}^{N}}|\widetilde{u}_k|^{2^{*}-1}\widetilde{u}.
\end{align*}
Passing to the limit we obtain
\begin{equation*}
\int_{\mathbb{R}^{N}}|\nabla \widetilde{u}|^{2}\leq \mu_{1}\int_{\mathbb{R}^{N}}|\widetilde{u}|^{2^{*}}.
\end{equation*}
Hence $\widetilde{s}\in (0,1]$. Similarly, since $\partial_{v}E_{\lambda_{k}}(u_{k},v_{k})v_{\infty}=0$, we get that $\widetilde{t}\in (0,1]$. Therefore,
\begin{align*}
c_{\lambda_k}^{G}&=\frac{1}{N}(\|u_k\|^2+\|v_k\|^2)\geq \left(\min_{x\in\bar{\Omega}} \#Gx\right)\frac{1}{N}\|\widetilde{u}\|^2+ \frac{1}{N}\|v_\infty\|^2 + o(1)\\
&\geq \left(\min_{x\in\bar{\Omega}} \#Gx\right)\frac{1}{N}\|\widetilde{s}\,\widetilde{u}\|^2+ \frac{1}{N}\|\widetilde{t}v_\infty\|^2 + o(1)
\end{align*}
and, passing to the limit, we get
\begin{equation} \label{eq:nontrivial}
\left(1+\min_{x\in\bar{\Omega}} \#Gx\right)\frac{1}{N}\mu_0^{\frac{2-N}{2}}S^{\frac{N}{2}}\geq \left(\min_{x\in\bar{\Omega}} \#Gx\right)\frac{1}{N}\|\widetilde{s}\,\widetilde{u}\|^2+ \frac{1}{N}\|\widetilde{t}v_\infty\|^2.
\end{equation}
As $\widetilde{s}\,\widetilde{u}$ belongs to the Nehari manifold associated to the problem
$$-\Delta u = \mu_1|u|^{2^*-2}u,\qquad u\in D^{1,2}(\mathbb{R}^N),$$
and $\widetilde{t}v_\infty$ belongs to the Nehari manifold associated to the problem
$$-\Delta v = \mu_2|v|^{2^*-2}v,\qquad u\in D^{1,2}_0(\Omega),$$
we have that
$$\|\widetilde{s}\,\widetilde{u}\|^2 \geq\mu_1^{(2-N)/2}S^{N/2}\qquad\text{and}\qquad \|\widetilde{t}v_\infty\|^2 > \mu_2^{(2-N)/2}S^{N/2}.$$
This contradicts the inequality \eqref{eq:nontrivial}. The proof of claim \eqref{eq:claim} is complete.

We define
\begin{equation*}
s^{2^{*}-2}:=\frac{\int_{\Omega}|\nabla u_{\infty}|^{2}}{\int_{\Omega}\mu_{1}|u_{\infty}|^{2^{*}}} \qquad\mbox{ and }\qquad t^{2^{*}-2}:=\frac{\int_{\Omega}|\nabla v_{\infty}|^{2}}{\int_{\Omega}\mu_{2}|v_{\infty}|^{2^{*}}}.
\end{equation*}
Since $u_\infty\geq 0$, $v_\infty\geq 0$ and $u_\infty v_\infty =0$, we have that $su_{\infty}- tv_{\infty}\in\mathcal{E}^{G}$. In addition,  since $\partial_{u}E_{\lambda_{k}}(u_{k},v_{k})u_{\infty}=0$ and $\partial_{v}E_{\lambda_{k}}(u_{k},v_{k})v_{\infty}=0$, arguing as above, we see that $s,t\in (0,1]$. So, using \eqref{eq:upper bound}, we get
\begin{align*}
c_\infty^G &\leq \frac{1}{N}(\|su_\infty\|^2 + \|tv_\infty\|^2)\leq \frac{1}{N}(\|u_\infty\|^2 + \|v_\infty\|^2) \\
&\leq \frac{1}{N}(\|u_\infty\|^2 + \|v_\infty\|^2) + \frac{1}{N}\lim_{k\to\infty}(\|u_k-u_\infty\|^2 + \|v_k-v_\infty\|^2) \\
&= \frac{1}{N}\lim_{k\to\infty}(\|u_k\|^2 + \|v_k\|^2)=\lim_{k\to\infty}E_{\lambda_k}(u_k,v_k)=\lim_{k\to\infty}c_{\lambda_k}^G\leq c_\infty^G.
\end{align*}
Therefore, $u_{k}\to u_{\infty}$ and $v_{k}\to v_{\infty}$ strongly in $D^{1,2}_0(\Omega)^{G}$, $s=t=1$, $u_{\infty}- v_{\infty}\in\mathcal{E}^{G}$ and
\begin{equation*}
J(u_{\infty}-v_{\infty})= \lim_{k\to\infty}c_{\lambda_{k}}^{G}= c_{\infty}^{G}.
\end{equation*}
The argument given in \cite[Lemma 2.6]{ccn} shows that $u_{\infty}-v_{\infty}$ is a critical point of $J$, i.e., $u_{\infty}-v_{\infty}$ is a sign-changing $G$-invariant solution of \eqref{eq:limitprob}. In particular, $u_{\infty}-v_{\infty}$ is continuous. Hence, $u_{\infty}=(u_{\infty}-v_{\infty})^+$ and $-v_{\infty}=(u_{\infty}-v_{\infty})^-$ are $G$-invariant and continuous. Consequently, the sets $\Omega_1:=\{x\in\Omega:u_\infty(x)>0\}$ and $\Omega_2:=\{x\in\Omega:v_\infty(x)>0\}$ are $G$-invariant and open in $\mathbb{R}^N$,  $\Omega_{1}\cap\Omega_{2}=\emptyset$, $\overline{\Omega_{1}\cup\Omega_{2}}=\overline{\Omega}$, $u_{\infty}$ solves the problem
$$-\Delta u=\mu_{1}|u|^{{2}^*-2}u,\qquad u\in D_{0}^{1,2}(\Omega_{1}),$$
and $v_{\infty}$ solves the problem
$$-\Delta v=\mu_{2}|v|^{{2}^*-2}v,\qquad v\in D_{0}^{1,2}(\Omega_{2}),$$
as claimed.
\end{proof}

\subsection{Multiple entire solutions}

Now we turn our attention to the competitive system \eqref{eq:system} in $\mathbb{R}^N$. We shall consider symmetries given by conformal transformations. We give a brief account of the symmetric setting. Details may be found in \cite[Section 3]{cp}.

Let $\mathscr{G}$ be a closed subgroup of $O(N+1)$. Then, $\mathscr{G}$ acts isometrically on the standard sphere $\mathbb{S}^N$. Using the stereographic projection $\sigma:\mathbb{S}^N \to \mathbb{R}^N\cup\{\infty\}$, we transfer this action to $\mathbb{R}^N$. Namely, for each $\mathit{g}\in\mathscr{G}$ we have a conformal transformation $\widetilde{\mathit{g}}:=\sigma\circ\mathit{g}^{-1}\circ\sigma^{-1}:\mathbb{R}^N\to\mathbb{R}^N$, which is well defined except at a single point. 

The space $D^{1,2}(\mathbb{R}^N)$ is a $\mathscr{G}$-Hilbert space with the action defined by
$$\mathit{g}u:=|\det \widetilde{\mathit{g}}'|^{1/2^*} u\circ\widetilde{\mathit{g}}, \qquad \mathit{g}\in\mathscr{G},\,u\in D^{1,2}(\mathbb{R}^N),$$
and $\mathscr{D}(\mathbb{R}^N)$ is also a $\mathscr{G}$-Hilbert space with the diagonal action $\mathit{g}(u,v):=(\mathit{g}u,\mathit{g}v)$. We set
$$\mathscr{D}(\mathbb{R}^N)^\mathscr{G}:=\{(u,v)\in \mathscr{D}(\mathbb{R}^N):\mathit{g}(u,v)=(u,v)\}.$$
It is easy to see that the functional $E$ is $\mathscr{G}$-invariant, and so are the functionals $F_1$ and $F_2$ defined in Proposition \ref{prop:nehari}. Hence, 
$$\mathcal{N}(\mathbb{R}^N)^\mathscr{G}:=\mathcal{N}(\mathbb{R}^N)\cap\mathscr{D}(\mathbb{R}^N)^\mathscr{G}$$
is a closed $\mathcal{C}^1$-submanifold of $\mathscr{D}(\mathbb{R}^N)^\mathscr{G}$ and a natural constraint for $E$.

The advantage of taking this kind of actions is that $O(N+1)$ contains subgroups $\mathscr{G}$ such that the $\mathscr{G}$-orbit of every point $p\in\mathbb{S}^N$ satisfies $0<\dim(\mathscr{G}p)<N$. For example, the group $\mathscr{G}=O(m)\times O(n)$ with $m+n=N+1$, $m,n\geq 2$, has this property, as, for this group, $\mathscr{G}p$ is homeomorphic to either $\mathbb{S}^{m-1}$, or to $\mathbb{S}^{n-1}$, or to $\mathbb{S}^{m-1}\times \mathbb{S}^{n-1}$. This property plays a role in the following lemmas.

\begin{lemma} \label{lem:PSentire}
If $\dim(\mathscr{G}p)>0$ for every $p\in\mathbb{S}^N$, then $E$ satisfies the $(PS)_c$-condition on $\mathcal{N}(\mathbb{R}^N)^\mathscr{G}$ for every $c\in\mathbb{R}$.
\end{lemma}

\begin{proof}
See \cite[Proposition 3.6]{cp}.
\end{proof}

\begin{lemma} \label{lem:genus_entire}
If $\dim(\mathscr{G}p)<N$ for every $p\in\mathbb{S}^N$, then the set
$$\Sigma_j^\mathscr{G}(\mathbb{R}^N):=\{Z\subset \mathcal{N}(\mathbb{R}^N)^{\mathscr{G}}:Z\text{ is symmetric and compact, and }\mathrm{genus}(Z)\geq j\}$$
is nonempty, for every $j\geq 1$.
\end{lemma}

\begin{proof}
If $\dim(\mathscr{G}p)<N$ for every $p\in\mathbb{S}^N$, then, for any given $j\geq 1$, there exist $2j$ pairwise disjoint open $\mathscr{G}$-invariant subsets of $\mathbb{R}^N$ (take, for example, $2j$ distinct $\mathscr{G}$-orbits and pairwise disjoint $\mathscr{G}$-invariant neighborhoods of them). Now we may argue as in the proof of Lemma \ref{lem:genus}. 
\end{proof}

Set
$$c_j^\mathscr{G}(\mathbb{R}^N):=\inf_{Z\in\Sigma_j^\mathscr{G}(\mathbb{R}^N)}\,\max_{(u,v)\in Z}E(u,v).$$

\begin{proof}[Proof of Theorem \ref{thm:entire} for $\lambda<0$.]
Let $\mathscr{G}$ be a closed subgroup of $O(N+1)$ such that $0<\dim(\mathscr{G}p)<N$ for every $p\in\mathbb{S}^N$. Then, by Lemmas \ref{lem:PSentire} and \ref{lem:genus_entire} and Theorem \ref{thm:szulkin}, we have that $c_j^\mathscr{G}(\mathbb{R}^N)$ is a critical value of the restriction of $E$ to $\mathcal{N}(\mathbb{R}^N)^\mathscr{G}$ for every $j\geq 1$. 

Moreover, as $E$ satisfies the $(PS)_c$-condition on $\mathcal{N}(\mathbb{R}^N)^\mathscr{G}$, the critical sets $K_c$ are compact and, hence, $\mathrm{genus}(K_c)<\infty$ for every $c\in\mathbb{R}$. It follows from Theorem \ref{thm:szulkin} that $\#\{c_j^\mathscr{G}(\mathbb{R}^N):j\geq 1\}=\infty$, i.e., $E$ has infinitely many critical values on $\mathcal{N}(\mathbb{R}^N)^\mathscr{G}$.

After replacing the minimizer $(u_1,v_1)$ of $E$ on $\mathcal{N}(\mathbb{R}^N)^\mathscr{G}$ with $(|u_1|,|v_1|)$, we get a positive critical point. The proof is complete.
\end{proof}

\appendix

\section{Appendix. Some results for the mixed term} \label{sec:bl}

\begin{lemma} \label{lem:mixed}
Let $\alpha,\beta \in [1,\infty)$. Given $\varepsilon>0$, there exists $C>0$ such that
\begin{align*}
& \left| |a_1+b_1|^{\alpha} |a_2+b_2|^{\beta} - |a_1|^{\alpha}|a_2|^{\beta} \right| \\ 
& \qquad \leq \varepsilon |a_1|^{\alpha}|a_2|^{\beta} + C\left(|a_1|^{\alpha}|b_2|^{\beta} + |b_1|^{\alpha}|a_2|^{\beta} + |b_1|^{\alpha}|b_2|^{\beta} \right),
\end{align*}
for all $a_1,a_2,b_1,b_2 \in \mathbb{R}$.
\end{lemma}

\begin{proof}
Fix $0<\delta<\min\{\varepsilon / (2^{\alpha}+1),\,1 \}$. Then, there exists $\bar{C}>0$ such that
\begin{align*}
& | |a_1+b_1|^{\alpha} |a_2+b_2|^{\beta} - |a_1|^{\alpha}|a_2|^{\beta} | \\ 
&\qquad \leq |a_1+b_1|^{\alpha} \, | |a_2+b_2|^{\beta} - |a_2|^{\beta} | + | |a_1+b_1|^{\alpha} - |a_1|^{\alpha} |\, |a_2|^{\beta} \\
&\qquad \leq 2^{\alpha}(|a_1|^{\alpha}+|b_1|^{\alpha})(\delta|a_2|^{\beta} + \bar{C}|b_2|^{\beta}) + (\delta|a_1|^{\alpha} + \bar{C}|b_1|^{\alpha})|a_2|^{\beta} \\
&\qquad \leq \varepsilon |a_1|^{\alpha}|a_2|^{\beta} + 2^{\alpha}\bar{C}|a_1|^{\alpha}|b_2|^{\beta} + (2^{\alpha}+\bar{C})|b_1|^{\alpha}|a_2|^{\beta} + 2^{\alpha}\bar{C}|b_1|^{\alpha}|b_2|^{\beta},
\end{align*}
as claimed.
\end{proof}

\begin{lemma} \label{lem:bl}
If $u_k \rightharpoonup u$ and $v_k \rightharpoonup v$ weakly in $D^{1,2}(\mathbb{R}^N)$, $u,v \in L^{\infty}_{\mathrm{loc}}(\mathbb{R}^N)$, $\alpha,\beta \in [1,\infty)$ and $\alpha + \beta = 2^*$, then, after passing to a subsequence,
$$\int_{\mathbb{R}^N} |u_k|^{\alpha}|v_k|^{\beta} - \int_{\mathbb{R}^N} |u_k-u|^{\alpha}|v_k-v|^{\beta} - \int_{\mathbb{R}^N} |u|^{\alpha}|v|^{\beta} = o(1).$$
\end{lemma}

\begin{proof}
After passing to a subsequence, we have that $u_k \to u$ and $v_k \to v$ a.e. in $\mathbb{R}^N$, $u_k \to u$ in $L^{\alpha}_{\mathrm{loc}}(\mathbb{R}^N)$ and $v_k \to v$ in in $L^{\beta}_{\mathrm{loc}}(\mathbb{R}^N)$.

Let $\varepsilon>0$, and fix $C>0$ as in Lemma \ref{lem:mixed}. Set
\begin{align}
w_k := &\left| |u_k|^{\alpha}|v_k|^{\beta} - |u_k-u|^{\alpha}|v_k-v|^{\beta} - |u|^{\alpha}|v|^{\beta} \right| \nonumber \\
&-\varepsilon|u_k-u|^{\alpha}|v_k-v|^{\beta} - C\left(|u|^{\alpha}|v_k-v|^{\beta} + |u_k-u|^{\alpha}|v|^{\beta} \right). \label{eq:w_k}
\end{align}
Then, $w_k \to 0$ a.e. in $\mathbb{R}^N$ and $w_k \in L^1(\mathbb{R}^N)$. Moreover, applying Lemma \ref{lem:mixed} with $a_1:=u_k-u, \,a_2=:v_k-v, \,b_1:=u,\,b_2:=v$, we get that $w_k \leq (C+1)|u|^{\alpha}|v|^{\beta}$. So, by the dominated convergence theorem,
\begin{equation} \label{eq:w+}
\lim_{k\to\infty} \int_{\mathbb{R}^N} w_k^+ =0,
\end{equation}
where $w_k^+ := \max\{w_k,0\}$. Fix $R>0$ large enough so that
\begin{align}
&C\int_{|x|\geq R} |u|^{\alpha}|v_k-v|^{\beta}\leq C|v_k-v|_{2^*}^\beta\left(\int_{|x|\geq R}|u|^{2^*}\right)^{\alpha / 2^*} < \varepsilon, \label{eq:exterior1} \\
&C\int_{|x|\geq R} |u_k-u|^{\alpha}|v|^{\beta}\leq C|u_k-u|_{2^*}^\alpha\left(\int_{|x|\geq R}|v|^{2^*}\right)^{\beta / 2^*} < \varepsilon, \label{eq:exterior2}
\end{align}
where $|\cdot|_{2^*}$ denotes the norm in $L^{2^*}(\mathbb{R}^N)$. Then, for $k$ large enough,
\begin{equation} \label{eq:interior1}
C\int_{|x| < R} |u|^{\alpha}|v_k-v|^{\beta} \leq C \max_{|x| \leq R}|u(x)|^{\alpha} \int_{|x| < R} |v_k-v|^{\beta} < \varepsilon,
\end{equation}
\begin{equation} \label{eq:interior2}
C\int_{|x| < R} |u_k-u|^{\alpha}|v|^{\beta} \leq C \max_{|x| \leq R}|v(x)|^{\beta} \int_{|x| < R} |u_k-u|^{\alpha} < \varepsilon.
\end{equation}
From \eqref{eq:w_k}, \eqref{eq:w+}, \eqref{eq:exterior1}, \eqref{eq:exterior2}, \eqref{eq:interior1} and \eqref{eq:interior2} we get that
\begin{align*}
&\int_{\mathbb{R}^N}\left| |u_k|^{\alpha}|v_k|^{\beta} - |u_k-u|^{\alpha}|v_k-v|^{\beta} - |u|^{\alpha}|v|^{\beta} \right| \\
&\leq \varepsilon\int_{\mathbb{R}^N}|u_k-u|^{\alpha}|v_k-v|^{\beta} + C\int_{\mathbb{R}^N}\left(|u|^{\alpha}|v_k-v|^{\beta} + |u_k-u|^{\alpha}|v|^{\beta}\right) + \int_{\mathbb{R}^N} w_k^+ \\
&<\tilde{C}\varepsilon
\end{align*}
for $k$ large enough, as claimed.
\end{proof}

Taking $u_k=v_k$ in Lemma \ref{lem:bl}, we obtain the well known Brezis-Lieb identity
$$\int_{\mathbb{R}^N} |u_k|^{2^*} - \int_{\mathbb{R}^N} |u_k-u|^{2^*} - \int_{\mathbb{R}^N} |u|^{2^*} = o(1).$$

\begin{lemma} \label{lem:auxiliary}
Let $\Theta$ be a bounded domain in $\mathbb{R}^N$. If $u_k \rightharpoonup u$ and $v_k \rightharpoonup v$ weakly in $D^{1,2}(\mathbb{R}^N)$, $u,v \in L^{\infty}_{\mathrm{loc}}(\mathbb{R}^N)$, $\alpha,\beta \in (1,\infty)$ and $\alpha + \beta = 2^*$, then, after passing to a subsequence, the following statements hold true:
\begin{itemize}
\item[$(a)$] If $q\in [1,2^*)$ and $\bar{q}:=\frac{q}{\beta -1}\geq 1$, then $\left||v_k|^{\beta}-|v_k-v|^{\beta}-|v|^{\beta}\right| \in L^{\bar{q}}(\Theta)$ and
$$\lim_{k\to\infty}\int_{\Theta} \left||v_k|^{\beta}-|v_k-v|^{\beta}-|v|^{\beta}\right|^{\bar{q}}=0.$$
\item[$(b)$] If $\alpha>2$, $p\in [1,2^*)$ and $\bar{p}:=\frac{p}{\alpha-2}\geq 1$, then 
$$\left||u_k|^{\alpha -2}u_k-|u_k-u|^{\alpha -2}(u_k-u)-|u|^{\alpha -2}u\right| \in L^{\bar{p}}(\Theta) \quad \text{and}$$
\begin{equation} \label{eq:alpha}
\lim_{k\to\infty}\int_{\Theta} \left||u_k|^{\alpha -2}u_k-|u_k-u|^{\alpha -2}(u_k-u)-|u|^{\alpha -2}u\right|^{\bar{p}}=0.
\end{equation}
If $\alpha\in (1,2]$, then $\left||u_k|^{\alpha -2}u_k-|u_k-u|^{\alpha -2}(u_k-u)-|u|^{\alpha -2}u\right| \in L^{\infty}(\Theta)$ and \eqref{eq:alpha} is true for every $\bar{p}\in [1,\infty)$.
\end{itemize}
\end{lemma}

\begin{proof}
Throughout the proof $C$ will denote a positive constant, not necessarily the same one.

$(a):$ Passing to a subsequence we have that $v_k\to v$ in $L^q(\Theta)$ and a.e. in $\Theta$. Using the mean value theorem we get that
\begin{align} 
\left||v_k|^{\beta}-|v_k-v|^{\beta}\right|&\leq \beta\left(|v_k-v|+|v|\right)^{\beta-1}|v| \nonumber \\
&\leq C\left(|v_k-v|^{\beta-1}|v|+|v|^\beta\right). \label{eq:meanvalue1}
\end{align}
Hence, as $v \in L^{\infty}(\Theta)$, we obtain
\begin{align*}
\left||v_k|^{\beta}-|v_k-v|^{\beta}-|v|^{\beta}\right|^{\bar{q}}&\leq C\left(|v_k-v|^{\beta-1}|v|+|v|^\beta\right)^{\bar{q}} \\
&\leq C|v_k-v|^q + C \qquad \text{a.e. in }\Theta. 
\end{align*}
This implies that $\left||v_k|^{\beta}-|v_k-v|^{\beta}-|v|^{\beta}\right| \in L^{\bar{q}}(\Theta)$. Set
$$w_k:=\left||v_k|^{\beta}-|v_k-v|^{\beta}-|v|^{\beta}\right|^{\bar{q}}-C|v_k-v|^q.$$
Then, $|w_k^+|\leq C$ and, by the dominated convergence theorem,
$$\lim_{k\to\infty}\int_{\Theta} \left||v_k|^{\beta}-|v_k-v|^{\beta}-|v|^{\beta}\right|^{\bar{q}}\leq \lim_{k\to\infty}\int_\Theta w_k^+ + \lim_{k\to\infty}C\int_\Theta |v_k-v|^q=0,$$
as claimed.

$(b):$ Passing to a subsequence we have that $u_k\to u$ in $L^p(\Theta)$ and a.e. in $\Theta$. From the mean value theorem we obtain
\begin{align}
&\left||u_k|^{\alpha -2}u_k-|u_k-u|^{\alpha -2}(u_k-u)\right| \leq(\alpha-1)\left(|u_k-u|+|u|\right)^{\alpha-2}|u| \nonumber \\
& \qquad \qquad\leq
\begin{cases} 
C\left(|u_k-u|^{\alpha-2}|u| + |u|^{\alpha-1}\right) &\text{if }\alpha>2 \\
(\alpha-1)|u|^{\alpha -1} &\text{if }\alpha\in (1,2]. \label{eq:meanvalue2}
\end{cases}
\end{align}
As $u\in L^\infty(\Theta)$, the statement follows immediately from the dominated convergence theorem if $\alpha\in (1,2]$. For $\alpha>2$ the argument is similar to the one we used to prove $(a)$.
\end{proof}

\begin{lemma} \label{lem:derivative}
If $u_k \rightharpoonup u$ and $v_k \rightharpoonup v$ weakly in $D^{1,2}(\mathbb{R}^N)$, $u,v \in L^{\infty}_{\mathrm{loc}}(\mathbb{R}^N)$, $\alpha,\beta \in (1,\infty)$ and $\alpha + \beta = 2^*$, then, after passing to a subsequence,
$$\int_{\mathbb{R}^N} |u_k|^{\alpha-2}|v_k|^{\beta}u_k - \int_{\mathbb{R}^N} |u_k-u|^{\alpha-2}|v_k-v|^{\beta}(u_k-u) - \int_{\mathbb{R}^N} |u|^{\alpha-1}|v|^{\beta}u = o(1),$$
$$\int_{\mathbb{R}^N} |u_k|^{\alpha}|v_k|^{\beta-2}v_k - \int_{\mathbb{R}^N} |u_k-u|^{\alpha}|v_k-v|^{\beta-2}(v_k-v) - \int_{\mathbb{R}^N} |u|^{\alpha}|v|^{\beta-2}v = o(1),$$
in $(\mathscr{D}(\mathbb{R}^N))'$.
\end{lemma}

\begin{proof}
It suffices to prove the first identity.

Set $f(t):=|t|^{\alpha-2}t$ and, for $R>0$,  let $B_{R}$ be the ball centered at $0$ of radius $R$ and $B_{R}^{c}$ be its complement in $\mathbb{R}^{N}$. Let $\varphi\in D^{1,2}(\mathbb{R}^{N})$. Then
\begin{align}
&\left|\int_{B_{R}^{c}} f(u_k)|v_k|^{\beta}\varphi - \int_{B_{R}^{c}} f(u_k-u)|v_k-v|^{\beta}\varphi - \int_{B_{R}^{c}} f(u)|v|^{\beta} \varphi\right| \nonumber \\
&\qquad\leq \int_{B_{R}^{c}}|u_{k}|^{\alpha-1}\left||v_{k}|^{\beta}-|v_{k}-v|^{\beta}\right||\varphi| +
\int_{B_{R}^{c}}|f(u_{k})-f(u_{k}-u)||v_{k}-v|^{\beta}|\varphi| \nonumber \\
&\qquad\qquad + \int_{B_{R}^{c}}|u|^{\alpha-1}|v|^{\beta}|\varphi|. \label{eq:ball complement}
\end{align}
Fix $\varepsilon>0$. From \eqref{eq:meanvalue1} we derive
\begin{align*}
&\int_{B_{R}^{c}}|u_{k}|^{\alpha-1}\left||v_{k}|^{\beta}-|v_{k}-v|^{\beta}\right||\varphi| \leq C\int_{B_{R}^{c}}|u_{k}|^{\alpha-1}\left(|v_k-v|^{\beta-1}|v|+|v|^\beta\right)|\varphi| \\
&\qquad\leq C|u_{k}|^{\alpha-1}_{2^*}\left(|v_k-v|^{\beta-1}_{2^*}+|v|^{\beta-1}_{2^*}\right)\left(\int_{B_{R}^{c}}|v|^{2^*}\right)^{1/2^*}|\varphi|_{2^*} \leq \frac{\varepsilon}{3} \|\varphi\|,
\end{align*}
if $R$ is large enough, where $|\cdot|_s$ stands for the norm in $L^s(\mathbb{R}^N)$. Similarly, from \eqref{eq:meanvalue2} we obtain that
$$\int_{B_{R}^{c}}|f(u_{k})-f(u_{k}-u)||v_{k}-v|^{\beta}|\varphi| \leq \frac{\varepsilon}{3} \|\varphi\|$$
for $R$ large enough. Clearly, the same is true for the last integral in \eqref{eq:ball complement}. Now, we fix $R>0$ such that
$$\left|\int_{B_{R}^{c}} f(u_k)|v_k|^{\beta}\varphi - \int_{B_{R}^{c}} f(u_k-u)|v_k-v|^{\beta}\varphi - \int_{B_{R}^{c}} f(u)|v|^{\beta} \varphi\right| \leq \varepsilon \|\varphi\|.$$
In $B_R$ we have that
\begin{align*}
&\left|\int_{B_{R}} f(u_k)|v_k|^{\beta}\varphi - \int_{B_{R}} f(u_k-u)|v_k-v|^{\beta}\varphi - \int_{B_{R}} f(u)|v|^{\beta} \varphi\right| \\
&\qquad\leq \int_{B_{R}} |f(u_k)|\left||v_k|^{\beta}-|v_k-v|^{\beta}-|v|^{\beta}\right||\varphi| \\
&\qquad\qquad + \int_{B_{R}} |f(u_k)-f(u_k-u)-f(u)||v_k-v|^{\beta}|\varphi| \\
&\qquad\qquad + \int_{B_{R}} |f(u_k)-f(u_k-u)-f(u)||v|^{\beta}|\varphi| \\
&\qquad\qquad + \int_{B_{R}} |f(u_k-u)||v|^{\beta}|\varphi| \; + \int_{B_{R}} |f(u)||v_{k}-v|^{\beta}|\varphi|.
\end{align*}
Now, we estimate the integrals on the RHS using Lemma \ref{lem:auxiliary}. For the first one, we fix $q\in [1,2^*)$ such that $\bar{q}:=\frac{q}{\beta -1}>1$ and $1-\frac{\alpha-1}{2^*}-\frac{\beta-1}{q}\geq\frac{1}{2^*}$. Then,
\begin{align*}
&\int_{B_{R}} |f(u_k)|\left||v_k|^{\beta}-|v_k-v|^{\beta}-|v|^{\beta}\right||\varphi| \\
&\qquad\leq C|u_k|^{\alpha-1}_{2^*}\left(\int_{B_R}\left||v_k|^{\beta}-|v_k-v|^{\beta}-|v|^{\beta}\right|^{\bar{q}}\right)^{1/\bar{q}}|\varphi|_{2^*} = \frac{\varepsilon}{5}\|\varphi\|,
\end{align*}
for large enough $k$. The other integrals are estimated in a similar way. This shows that
$$\left|\int_{\mathbb{R}^N} f(u_k)|v_k|^{\beta}\varphi - \int_{\mathbb{R}^N} f(u_k-u)|v_k-v|^{\beta}\varphi - \int_{\mathbb{R}^N} f(u)|v|^{\beta} \varphi\right| \leq 2\varepsilon \|\varphi\|,$$
for large enough $k$, and finishes the proof of the lemma.
\end{proof}

Taking $u_k=v_k$ in Lemma \ref{lem:derivative}, we obtain the well known identity
$$\int_{\mathbb{R}^N} |u_k|^{2^*-2}u_k - \int_{\mathbb{R}^N} |u_k-u|^{2^*-2}(u_k-u) - \int_{\mathbb{R}^N} |u|^{2^*-2}u = o(1)$$
in $(D^{1,2}(\mathbb{R}^N))'$; see \cite[Lemma 8.9]{w}.

 \vspace{10pt}

\begin{flushleft}
\textbf{Mónica Clapp}\\
Instituto de Matemáticas\\
Universidad Nacional Autónoma de México\\
Circuito Exterior, Ciudad Universitaria\\
04510 Coyoacán, CDMX\\
Mexico\\
\texttt{monica.clapp@im.unam.mx} \vspace{10pt}

\textbf{Jorge Faya}\\
Instituto de Matemáticas\\
Universidad Nacional Autónoma de México\\
Circuito Exterior, Ciudad Universitaria\\
04510 Coyoacán, CDMX\\
Mexico\\
\texttt{jorgefaya@gmail.com}
\end{flushleft}

\end{document}